%% file: arxiv_v2.tex
\documentclass[11pt]{amsart}

\usepackage{setspace, longtable}
\onehalfspace

\usepackage{amsmath,amsthm,amssymb,color,tikz,epsfig,amsfonts,latexsym,  bbm, mathrsfs}
\usepackage{epsfig,amsfonts,latexsym, MnSymbol}
\usepackage{hyperref}
\hypersetup{
    colorlinks=true,
    linkcolor=blue,
    citecolor=blue,
    urlcolor=blue,
    pdfborder={0 0 0}
}
\usepackage[comma, sort&compress,numbers]{natbib}

\usepackage{graphicx}
\usepackage{xcolor}

\usepackage[left=1in,top=1in,right=1in]{geometry}
\usepackage[font=sf, labelfont={sf,bf}, margin=1cm]{caption}

\numberwithin{equation}{section}
\usetikzlibrary{matrix}

\usepackage{geometry}

\usepackage{bbm}

\allowdisplaybreaks

\include{notation}

\begin{document}
\title[Branching random walk] {\vspace{-2cm}\textbf{Branching Random Walks, Stable Point Processes and Regular Variation} \vspace{0.5cm}}

\author[A. Bhattacharya]{Ayan Bhattacharya} \thanks{Ayan Bhattacharya's research was partially supported by the project RARE-318984 (a Marie Curie FP7 IRSES Fellowship).}
%\address{Statistics and Mathematics Unit, Indian Statistical Institute, 203 B. T. Road, Kolkata 700108.}

\email{ayanbhattacharya.isi@gmail.com, rajatmaths@gmail.com, parthanil.roy@gmail.com}

\author[R. S. Hazra]{Rajat Subhra Hazra}\thanks{Rajat Subhra Hazra's research was supported by Cumulative Professional Development Allowance from Ministry of Human Resource Development, Government of India and Department of Science and Technology, Inspire funds.}
%\address{Statistics and Mathematics Unit, Indian Statistical Institute, 203 B. T. Road, Kolkata 700108.}
%\email{rajatmaths@gmail.com}

\author[P. Roy]{Parthanil Roy}\thanks{Parthanil Roy's research was supported by Cumulative Professional Development Allowance from Ministry of Human Resource Development, Government of India and the project RARE-318984 (a Marie Curie FP7 IRSES Fellowship).}
%\address{Statistics and Mathematics Unit, Indian Statistical Institute, 203 B. T. Road, Kolkata 700108}
%\email{parthanil.roy@gmail.com}



\maketitle
\centerline{Statistics and Mathematics Unit, Indian Statistical Institute, Kolkata}

\begin{abstract}
Using the language of regular variation, we give a sufficient condition for a point process to be in the superposition domain of attraction of a strictly stable point process. This sufficient condition is then used to obtain  the weak limit of a sequence of point processes induced by a branching random walk with jointly regularly varying displacements. Because of heavy tails of the step size distribution, we can invoke a one large jump principle at the level of point processes to give an explicit representation of the limiting point process. As a consequence, we extend the main result of \citet{durrett:1983} and verify that two related predictions of \citet{brunet:derrida:2011} remain valid for this model.
\end{abstract}

%\setcounter{tocdepth}{1}
%\tableofcontents

%\allowdisplaybreaks

\section{Introduction}\label{sec:introduction}
Branching random walk on the real line can be described as follows. In the zeroth generation, one particle is born at the origin. It branches into a number of offspring particles and positions them according to a point process $\pointl$ on the real line giving rise to the first generation. Each of the particles in the first generation produces offspring and they (the offsprings) undergo displacements (with respect to the positions of their parents) according to independent copies of the same point process $\pointl$. The position of a particle in the second generation is its displacement translated by its parent's position.  This forms the second generation, and so on. Assume further that the random number of new particles produced by a particle and the displacements corresponding to the new particles are independent. The resulting system is known as a branching random walk.

Let $Z_n$ denote the number of particles in the $n^{th}$ generation. Clearly $\{Z_n\}_{n \geq 0}$ forms a Galton-Watson branching process with $Z_0 \equiv 1$. We assume that this branching process is supercritical and condition on its survival. Further in this article, the displacements of offspring particles coming from the same parent will be dependent and multivariate regularly varying. We shall investigate this model from the point of view of extreme value theory (see Theorem \ref{thm:mainthm2}) and extend the work of \citet{bhattacharya:hazra:roy:2014}. In particular, this answers a question of Antar Bandyopadhyay and Jean Bertoin (asked independently during personal communications with the first author).

The earliest works on branching random walks include \citet{hammersley1974}, \citet{kingman1975}, \citet{biggins1976}, etc. This model and its extreme value theory have now become very important because of their connections to various probabilistic models (e.g., Gaussian free fields, conformal loop ensembles, multiplicative cascades,  tree polymers etc.); see \citet{BZ2012}, \citet{hu:shi:2009, addario2009}, \citet{aidekon2013},
\citet{biskup:louidor:2013, biskup:louidor:2014}, \citet{bramson:ding:zeitouni:2013}, \citet{dey:waymire:2015}. For existing results on branching random walks with heavy-tailed displacements and their continuous parameter analogues, see \citet{durrett:1979, durrett:1983}, \citet{kyprianou:1999}, \citet{gantert2000}; see also \citet{lalley:shao:2013}, \citet{berard2014} and \citet{maillard:2015} for the latest developments in this direction.

%Branching random walk has recently gained further prominence in recent years due to connections to various models. The examples include multiplicative cascades, Gaussian free fields, conformal loop ensembles to name a few. The study of extreme values of branching random walk and branching brownian motion has been studied extensively in recent years. The point process  of branching brownian motion was studied by \citet{arguin2012poissonian, arguin2013extremal} and \citet{aidekon2013branching}. The limiting point process of branching random walk under above light tailed assumptions was shown by \citet{madaule:2011}. The maximum and point process of displacements in case regularly varying i.i.d.\ displacements was derived in \citet{durrett:1983} and \citet{bhattacharya:hazra:roy:2014} respectively. The limiting point process in case of branching random walk with light tailed assumptions and branching Brownian motion turned out to be a decorated Poisson process as predicted by \citet{brunet:derrida:2011}. The case of branching random walk by \citet{madaule:2011} used crucially the idea of superposability in \citet{maillard:2013} (see also \citet{davydov:molchanov:zuyev:2008}). In case i.i.d.\ regularly varying set-up it is known that the limiting process is a randomly scaled scale-decorated Poisson point process.

The purpose of this article is two-fold - we show that, for jointly regularly varying displacements, the extremal point process converges to a \emph{randomly scaled scale-decorated Poisson point process} and also find an explicit representation of the limiting point process. To this end, we study the stability property (as introduced by \citet{davydov:molchanov:zuyev:2008}) of the limiting point process and relate it to the regular variation of point processes (in the sense of \citet{hult:lindskog:2006}) based on heavy-tailed analogues of the main results of \citet{subag:zeitouni:2014}. Our mode of proof gives a mathematical justification behind obtaining a scale-decorated  Poisson point process in the limit. We also extend the result of \citet{durrett:1983} and show that, as in light-tailed case, the asymptotic position of the rightmost point is not qualitatively affected by the presence of dependence. 

This article is organised as follows. In Section~\ref{sec:prelim_main}, we present the background, develop the notations and state the main results in this paper. These results are proved in Sections~\ref{sec:stablepp} and \ref{sec:brw}, and their consequences are given in Section~\ref{sec:conseq}. Finally, we list all the important notions and notations used in this paper in the appendix.

%\been
%
%\item $\real_0 = \real \setminus \{0\}$.
%
%\item $\bbs =$ Polish space and $\bbs_0 = \bbs \setminus \{s_0\}$.
%
%\item $0_t \in \real^t$ is such that all its components are $0$ for all $t \in \bbn \cup \{\infty\}$.
%
%\item $\scrm_0 = \scrm(\real_0) \setminus \{ \nullm \}$ where $\nullm$ denotes the null-measure on $\real$.
%\een

\section{Preliminaries and Main Results} \label{sec:prelim_main}

In this section, we present the main results of this paper. To this end, we need to introduce some notations and develop some machineries. This is done by brief discussions of the key phrases used in the title of this paper in the reverse order. The connection between these notions will be clear when the main theorems are stated. All the random quantities defined in this paper are defined on a common probability space $(\Omega, \mathcal{F}, \prob)$ unless specified otherwise.

\subsection{Regular Variation} \label{subsec:reg_var}
The definition of regular variation on $\mathbb{R}^d$ is typically given based on vague convergence on the compactified and punctured space $[-\infty, \infty]^d \setminus \{0\}$; see, e.g., \citet{resnick:1987}. However, this method is not at all robust to spaces that are not locally compact (e.g., $\mathbb{R}^\mathbb{N}$, function spaces, spaces of measures, etc.), because compactification of such spaces leads to a number of topologically undesirable consequences; see \citet{hult:lindskog:2006} and \citet{lindskog:resnick:roy:2014} for a detailed discussion. In order to circumvent this obstacle, \citet{hult:lindskog:2006} introduced a general definition of regular variation with very mild conditions on the underlying space, as described below.

Let $(\bbs,d)$ be a Polish space and $\bullet: (0,\infty)\times \bbs \to \bbs$ be a continuous scalar multiplication satisfying $1 \bullet s = s$ for all $s \in \bbs$, and $b_1 \bullet (b_2 \bullet s) = (b_1b_2) \bullet s$ for every $b_1, b_2 >0$. Fix an element $s_0 \in \bbs$ such that $b \bullet s_0 = s_0$ for all $b >0$, and endow the space $\bbs_0 = \bbs \setminus \{s_0\}$ \label{not:s_0} with the relative topology. Let $\bbm(\bbs_0)$ denote the class of all Borel measures on $\bbs_0$ whose restrictions to $\bbs \setminus B(s_0,r)$ is finite for every $r>0$, and  let $\scrc_0$ denote the class of all bounded continuous functions $f:\bbs_0 \to [0,\infty)$ that vanish on $B(s_0,r) \setminus \{s_0\}$ for some $r>0$. We say that a sequence of measures $\{\nu_n\} \subseteq \bbm(\bbs_0)$ converges in the Hult-Lindskog (HL) sense to a measure $\nu \in \bbm(\bbs_0)$ (denoted by $\nu_n \hlconv \nu$) if
\aln{
\lim_{n \to \infty} \int_{\bbs_0} f (x)\dtv \nu_n(x) = \int_{\bbs_0} f (x)\dtv \nu (x)\mbox{ for every } f \in \scrc_0. \label{eq:hlconv}
}

\begin{defn}[{\bf Regular variation}; \citet{hult:lindskog:2006}] \label{defn:hlregvar}
A measure $\nu \in \bbm(\bbs_0)$ is regularly varying if there exists an $\alpha >0$, an increasing sequence of positive real numbers $\{b_n\}$ satisfying $b_{[\beta n]} / b_n \to \beta^{1/\alpha}$ for all $\beta > 0$, and a non-null measure $\lambda \in \bbm(\bbs_0)$ such that $n \nu( b_n \bullet\cdot) \hlconv \lambda(\cdot)$ as $n \to \infty$. This will be denoted by $ \nu \in \regvar(\bbs_0, \alpha, \lambda)$. \label{not:regvar_ams}
\end{defn}

From Theorem 3.1 in \citet{lindskog:resnick:roy:2014}, it is clear that the limit measure $\lambda \in \bbm(\bbs_0)$ satisfies the following scaling property:
\aln{
\lambda( b \bullet \cdot) = b^{-\alpha} \lambda(\cdot) \label{eq:homregvar}
}
for all $b>0$. This definition of regular variation boils down to the usual definition of regular variation on $\real$ or $\real^d$ as pointed out in Subsection 2.3 in \citet{lindskog:resnick:roy:2014}. In this article, we shall be interested in  regular variations on the spaces $\real^\bbn=\{\mathbf{x}=(x_1, x_2, \ldots): x_1, x_2, \ldots \in \mathbb{R}\}$ and $\scrm(\bar{\real}_0) = \{ \point : \point \mbox{ is a Radon point measure on } [-\infty,\infty]\setminus\{0\}\}$ \label{not:mr_0} as illustrated in the following examples.

\begin{example}\label{example:regvarprocess}
 Let us consider $\bbs = \real^\bbn$ and $s_0 = \bld{0}_\infty$, the zero element in $\real^\bbn$. For simplicity, consider the i.i.d.\ process $\bld{X} = \{X_n : n \in \bbn\}$ such that $\prob_{X_1}(\cdot) = \prob(X_1 \in \cdot) \in \regvar(\zrer=[-\infty,\infty]\setminus \{0\}, \alpha, \nu_\alpha)$, where
\aln{
\nu_\alpha(\dtv x) = \alpha p x^{-\alpha -1} \bbo_{(x > 0)} \dtv x + \alpha q (-x)^{-\alpha -1} \bbo_{(x <0)} \dtv x \label{eq:nualpha}
}
with $p, q \geq 0$, $p+q =1$. In particular, this means that the usual tail-balancing conditions hold, i.e.,
\aln{
\lim_{x \to \infty} \frac{\prob(X_1>x)}{\prob(|X_1|>x)} = p ~~~~ \mbox{ and }~~~~~~ \lim_{x \to \infty} \frac{\prob(X_1<-x)}{\prob(|X_1|>x)} = q\,. \label{eq:tailbalancing}
}
If $q =0$, then the measure $\nu_\alpha$ is supported on $(0,\infty)$ and is denoted by $\malpha$. It has been established in \citet{lindskog:resnick:roy:2014} that
\aln{
\prob_{\bld X} (\cdot) = \prob(\bld{X} \in \cdot) \in \regvar(\rnz = \real^\mathbb{N}\setminus\{\bld{0}_\infty\}, \alpha, \lamiid), \label{eq:iidprocesslim}
}
where $\lamiid$ is a measure on $\rnz$ (concentrating on the axes) such that
\aln{
\lamiid(\dtv \mathbf{x}) := \sum_{i=1}^\infty \prod_{j=1}^{i-1} \delta_0 (\dtv x_j) \times \nu_\alpha (\dtv x_i) \times \prod_{j=i+1}^\infty \delta_0(\dtv x_j). \label{eq:iidprocess_lambda}
}
Here $\delta_0$ denotes the Dirac measure putting unit mass at $0$. Examples where the limit measures are not concentrated on the axes were considered in \citet{resnick:roy:2014}. They investigated the corresponding regular variation property for stationary moving average processes with positive regularly varying innovations and positive coefficients and computed the limit measure explicitly.

%We also want to consider the moving average processes. For moving average processes the limit measure is computed in \citet{resnick:roy:2014}. We shall qoute the result from the above mentioned whenever we need it. The most important fact to notice in this case, because of dependence structure in moving average processes more than one large jump occurs. The formalize this we need the notion of HL-convergence.
\end{example}

\begin{example} \label{example:regvarpointmeasure}
Consider the Polish space $\bbs = \scrm (\bar{\real}_0)$ of all Radon point measures on $\bar{\real}_0:=[-\infty, \infty]\setminus\{0\}$ \label{not:r_0} endowed with the vague topology, and $s_0 = \nullm$ (the null measure). \label{not:nullm} The scalar multiplication by $b > 0$ is denoted by $\mbfs_b$ and is defined as follows: if $ \point =\sum_i \delta_{u_i} \in \scrm(\bar{\real}_0)$, then
\begin{equation}\label{eq: scalar}
\mbfs_b \point = b \bullet \point =\sum_{i} \delta_{b u_i}.
\end{equation}
In other words, a scalar multiple of a point measure is obtained by multiplying each point of the measure by a positive real number. The HL convergence in $\scrm_0 = \scrm(\bar{\real}_0) \setminus \{\nullm\}$ \label{not:scrm_0} has been discussed and used by \citet{hult:samorodnitsky:2010} (see also \citet{fasen:roy:2014}) in the context of large deviations. This convergence (more specifically,  Definition \ref{defn:hlregvar}) gives rise to the notion of regular variation for point processes, which will play crucial role in this paper.
\end{example}

\subsection{Strictly Stable Point Processes}

A point process on $\zrer$ is an $\scrm (\bar{\real}_0)$-valued random variable defined on $(\Omega, \mathcal{F}, \prob)$ that does not charge any mass to $\pm \infty$. The following definition of strict stability for such point processes was introduced in \citet{davydov:molchanov:zuyev:2008} and will be shown to be intimately connected to the regular variation on the space $\scrm_0$.

\begin{defn}[{\bf $\stas$ point process}; \citet{davydov:molchanov:zuyev:2008}] A point process $N$ (on $\zrer$) is called a strictly $\alpha$-stable ($\stas$) \label{not:stas} point process ($\alpha >0$) if for every $b_1, b_2 >0$,
\aln{
\mbfs_{b_1} N_1 + \mbfs_{b_2} N_2 \eqd \mbfs_{(b_1^\alpha + b_2^\alpha)^{1/\alpha}} N, \label{eq:alphastablepp}
}
where $N_1, N_2$ are independent copies of $N$, $+$ denotes superposition of point processes and $\eqd$ denotes equality in distribution.
\end{defn}

 The (sum) domain of attraction of $\alpha$-stable random variables and vectors is closely related to the notion of regular variation on $\real$ and $\real^d$, respectively; see, e.g., \citet{feller:1971} \citet{meerschaert:scheffler:2001}. The corresponding question has been investigated for normed cone-valued strictly $\alpha$-stable random variables in Subsection~4.4 of \citet{davydov:molchanov:zuyev:2008}. However, for $\stas$ point processes, this question has remained open. In this work, we fill  this gap partially and obtain a sufficient condition for a point process to belong to the (superposition) domain of attraction of an $\stas$ point process. This sufficient condition is given in terms of regular variation of the original point process, as described below.

 A point process $\pointl$ is called regularly varying if $\prob_{\pointl}(\cdot) = \prob(\pointl \in \cdot) \in \regvar(\scrm_0, \alpha , m^*)$ for some $\alpha >0$ and for some $m^* \in \bbm(\scrm_0)$. By a standard abuse of notation, we shall denote this by $\pointl \in \regvar(\scrm_0, \alpha , m^*)$. The following equivalence is our first main result, which is somewhat expected albeit nontrivial.

\begin{thm}\label{thm:mainthm1}
Let $\pointl$ be a point process on $\zrer$, and $\pointl_i$'s be independent copies of $\pointl$. If there exists a non-null measure $m^*$ on $\scrm_0$ such that $\pointl \in \regvar(\scrm_0, \alpha , m^*)$, then $\pointl$ is in the domain of attraction of an $\stas$ point process $\pointq$, i.e.
$\displaystyle{ \mbfs_{b_n^\inv}\sum_{i=1}^n \pointl_i \Rightarrow \pointq\,. }$
Furthermore, in the above situation, the Laplace functional of the limiting point process $\pointq$ is given by
\[
\exptn \Big( \exp \Big\{- \int f \dtv \pointq \Big\} \Big) = \exp \bigg\{ - \int_{\scrm_0} \Big(1- \exp\{-\int f \dtv \nu \} \Big)\, m^*(\dtv \nu) \bigg\}
\]
for all nonnegative real-valued measurable functions $f$ defined on $\bar{\real}_0$.
\end{thm}

\noindent Heavy-tailed analogues of a few results in \citet{subag:zeitouni:2014} form the building block of the proof of the above theorem, which in turn becomes significant in establishing the results on limiting point process induced by branching random walks with regularly varying displacements.

It was established in \citet{davydov:molchanov:zuyev:2008} that a point process (on $\real$) is strictly $\alpha$-stable if and only if it admits a series representation of a special kind. Motivated by the works of \citet{brunet:derrida:2011}, \citet{maillard:2013} and \citet{subag:zeitouni:2014}, this has been termed  a \emph{scale-decorated Poisson point process} (ScDPPP) representation in \citet{bhattacharya:hazra:roy:2014}. The precise form of this representation is given in the following definition.

\begin{defn}[{\bf Scale-decorated Poisson point process}]\label{defn:ScDPPP}
A point process $N$ is called a scale-decorated Poisson point process with intensity measure $m$ and scale-decoration $\point$ (denoted by $N \sim ScDPPP(m, \point)$) \label{not:scdppp} if there exists a Poisson random measure $\Lambda= \sum_{i=1}^\infty \delta_{\lambda_i}$  on $(0, \infty)$ with intensity measure $m$ and a point process $\point$ such that
$$N \eqd \sum_{i=1}^\infty \mbfs_{\lambda_i} \point_i,$$
where  $\point_1, \point_2, \ldots$ are independent copies of the point process $\point$ independent of $\Lambda$.
\end{defn}
\noindent As mentioned above, it is established in \citet{davydov:molchanov:zuyev:2008} (see Example~8.6 therein) that a point process $N$ is strictly $\alpha$-stable if and only if $N \sim ScDPPP(m_\alpha, \point)$ for some point process $\point$ (here $m_\alpha$ is as described in Example~\ref{example:regvarprocess}). The light-tailed analogue of this result has been proved in a novel approach by \citet{maillard:2013}.

\citet{bhattacharya:hazra:roy:2014} also introduced the following slightly more general notion in parallel to \citet{subag:zeitouni:2014}. A point process $M$ is called a {\it randomly scaled scale-decorated Poisson point process} (SScDPPP) with intensity measure $m$ and scale-decoration $\point$ and random scale $U$ (denoted by $N \sim ScDPPP(m, \point, U)$) if $M \eqd \mbfs_U N$ where $N \sim SScDPPP(m, \point)$ \label{not:sscdppp} and $U$ is a positive random variable independent of $N$. As we shall see in the next subsection, these randomly scaled strictly stable point processes arise as limits of point processes induced by branching random walks with regularly varying displacements.

\subsection{Branching Random Walks} \label{subsec:brw}

%Branching random walk on the real line can be described as follows. In the zeroth generation, one particle is born at the origin. It branches into a number of offspring particles and positions them according to a point process $\pointl$ on the real line giving rise to the first generation. Each of the particles in the
%first generation produces offsprings and they (the offsprings) undergo displacements (with respect to the positions of their parents) according to independent copies of the same point process $\pointl$. This forms the second generation, and so on. Assume further that each particle reproduces independently of each other and of all the displacements. The resulting system is known as a branching random walk.
First we recall that for a branching random walk (defined in Section~\ref{sec:introduction}), $Z_n$ denotes the number of particles in the $n^{th}$ generation ($Z_0 \equiv 1$),  and the sequence $\{Z_n: n \ge 0 \}$ forms a Galton-Watson branching process. We make some assumptions on the branching mechanism and the displacements,  as follows.
\begin{ass} \label{assump:model}
\emph{In our model, the point process $\pointl$ is of the form
\beq \label{eq:basicpointpro}
\pointl \overset{d}= \sum_{i=1}^{Z_1} \delta_{X_i},
\eeq
where $Z_1$ (the branching random variable) is as above and $\mathbf{X}= (X_1, X_2, \ldots)$ (displacement process) is a random element on the space $\real^\bbn$ independent of $Z_1$. We make the following assumptions on the displacements and the branching mechanism.}
\begin{enumerate}
\item \emph{\textbf{Assumptions on Displacements:} $X_1, X_2, \ldots$ are identically distributed with $\prob_{X_1}=\prob(X_1 \in \cdot) \in \regvar(\zrer, \alpha, \nu_\alpha) $, where $\nu_\alpha$ is as in \eqref{eq:nualpha}, and there exists a non-null measure $\lambda$ on $\rnz$ such that
 \aln{
\prob_{\bld{X}}=\prob(\bld{X} \in \cdot) \in \regvar (\rnz, \alpha, \lambda). \label{eq:jtregvar}
 }}
 \item \emph{\textbf{Assumptions on Branching Mechanism:} The underlying Galton-Watson process is supercritical with finite progeny mean, i.e., $\mu:=\exptn(Z_1) \in (1,\infty)$. Using the martingale convergence theorem, it is easy to see that there exists a non-negative random variable $W$, such that
\aln{
\frac{Z_n}{\mu^n} \to W \mbox{ almost surely} \label{eq:martngle_conv}
}
as $n \to \infty$. We shall further assume that the random variable $Z_1$ satisfies the Kesten-Stigum condition, i.e.,
\beq \label{eq:kestenstigum}
\exptn(Z_1 \log^+ Z_1) < \infty
\eeq
so that $W>0$ almost surely when conditioned on the survival of the tree.}
\end{enumerate}
\end{ass}

Let $\cals$ be the event that the underlying Galton-Watson tree survives to become an infinite tree. The conditional probability $\prob(\cdot\,|\,\cals)$ is denoted by $\pstar$ and the corresponding expectation operator is denoted by $\estar$. Let $\{\calf_n\}$ denote the natural filtration for the underlying Galton-Watson process.

%Let $\scrm(\rnz)$ denotes the space of all point measures on the space $\rnz$ which charges finite mass for every measurable subset of $\rnz$ bounded away from the element $\zfty$. To be precise, we shall assume that there exists an increasing sequence of constants $\{c_n\}$, such that $c_n \to \infty$ as $n \to \infty$, the marginal distribution of $\mathbf{X}$ has balanced regularly varying tails i.e.
%\beq \label{eq:marginalregvar}
%\prob( |X_i| > x ) = x^{-\alpha} L(x)
%\eeq
%for all $x>0$, where $\alpha >0$, and $L(x)$ is  a slowly varying function and the tail balance conditions
%\beq \label{eq:tailbalance}
%\frac{\prob(X_i > x)}{\prob(|X_i| >x)} \to p \mbox{ and } \frac{\prob(X_i < -x)}{\prob(|X_i| >x)} \to q
%\eeq
%as $x \to \infty$ for some $p,q \ge 0$ with $p+q =1$ and $\mathbf{X}=(X_1, X_2, \ldots)$ is regularly varying as an element of the space $\real^\bbn$ i.e.
%\beq \label{eq:jtregvar}
%\mu^n \prob( c_n^\inv \mathbf{X} \in \cdot  ) \hlconv \lambda( \cdot )
%\eeq
%on the space $\scrm_p(\rnz)$.

Branching random walk can also be viewed as a collection of random variables indexed by the underlying Galton-Watson tree $\bbt=(\mathbb{V}, \mathbb{E})$ as follows. To each edge $\uu \in \mathbb{E}$, attach the displacement random variable $X_{\uu}$ of the corresponding offspring particle. For each $v \in \mathbb{V}$, let $\Iv$ denote the unique geodesic path from the root $o$ to $\uv$, and let $|\uv|$ denote the generation of $\uv$. \label{not:gen_uv} $\Sv$ denotes the position of the particle corresponding to $\uv$. Then, clearly, $\Sv = \sum_{\uu \in \Iv} \Xu$.
%\beq \label{eq:displacement}
%\Sv = \sum_{\uu \in \Iv} \Xu.
%\eeq
The collection $\{\Sv\}_{\uv \in \mathbb{V}}$ is the branching random walk with $\{\Sv\}_{|\uv|=n}$ forming the $n^{th}$ generation.  Note that \eqref{eq:jtregvar} implies that there exists an increasing sequence $\{c_n\} = \{b_{[\mu^n]}\}$ (where $b_n$ is as in Definition~\ref{defn:hlregvar}) of positive real numbers such that
\aln{
\mu^n \prob(c_n^\inv \bld{X} \in \cdot) \hlconv \lambda(\cdot) \label{eq:jtregvar_cn}
}
in $\mathbb{M}(\rnz)$ as $n \to \infty$. We are interested to find the weak limit (under $\pstar$) of the point process sequence
\beq \label{eq:basicseqpoint}
N_n = \sum_{|\uv|=n} \delta_{c_n^\inv \Sv}, \;\; n \geq 1
\eeq
of properly normalized positions of the $n^{th}$ generation particles.

To describe the limiting point process, we need to introduce some more notations as follows. Let $\poi$ be a Poisson random measure
\aln{
\poi  = \sum_{l\ge 1} \delta_{(\xi_{l1}, \, \xi_{l2}, \,\ldots)} =: \sum_{l\ge 1} \delta_{\pmb{\xi}_l} \label{eq:defn_poi}
}
on $\rnz$ with intensity measure $\lambda$ and independent of $W$.  Let $V$ be a positive integer-valued random variable with probability mass function
\aln{
\prob(V=v) = \frac{1}{s} \prob(Z_1 =v) \sum_{i=0}^\infty \frac{1}{\mu^i} \Big(1- \prob(Z_i =0)^v \Big), \;\; v \in \bbn, \label{eq:pmf_V1}
}
where $s$ is the normalising constant.
%\aln{
 %s = \sum_{t=1}^\infty   \prob(Z_1 =t) \sum_{i=0}^\infty \frac{1}{\mu^i} \Big(1- \Big(\prob(Z_i =0)\Big)^t \Big).                                 \label{eq:pmf_V2}
%}
Suppose that $\bld{T}$ is an $\bbn_0^\bbn$-valued random variable and its probability mass function conditioned on $V$ is given as follows:
\aln{
\prob(\bld{T}= \bld{y} | V=v) = \begin{cases} 0 & \mbox{ if } y_k>0  \mbox{ for some } k>v \mbox{ or } \bld{y} =\bld{0}, \\
\frac{1}{s_v} \sum_{i=0}^\infty \frac{1}{\mu^i} \prod_{m=1}^v \prob(Z_i = y_m ) & \mbox{ otherwise, }\end{cases} \label{eq:cond_pmf_T1}
}
where $\bld{y}=(y_1, y_1, \ldots) \in \bbn^\infty$, $v \in \bbn$, and $s_v$ is the normalising constant.
%\aln{
%s_t = \sum_{(y_1, \ldots, y_t) \neq \bld{0}_t} \sum_{i=0}^\infty \frac{1}{\mu^i} \prod_{m=1}^t  \prob(Z_i =y_m) \label{eq:cond_pmf_T2}
%}
%with $\bld{0}_t$ being the zero element in $\bbn^t$.
Finally, we take a collection $\{(V_l, \bld{T}_l) : l \in \bbn\}=\{(V_l, (T_{l1}, T_{l2}, \ldots)) : l \in \bbn\}$ of independent copies of $(V, \bld{T})$ that is independent of $W$ and $\poi$. With these notations, we are now ready to state our second main result.

\begin{thm} \label{thm:mainthm2}
Suppose that Assumptions~\ref{assump:model} hold and consider the point process sequence $\{N_n\}$ defined by \eqref{eq:basicseqpoint} with $c_n$ as in \eqref{eq:jtregvar_cn}. Under $\pstar$, $N_n$ converges weakly (as $n \to \infty$) to the point process
\begin{equation}
 N_* \eqd \sum_{l=1}^\infty \sum_{k=1}^{V_l} T_{lk} \, \delta_{( s \mu^\inv W )^{1/\alpha} \xi_{lk}} \label{eq:rep_limit_measure}
\end{equation}
in the space $\scrm(\zrer)$. Moreover, the limiting point process $N_*$ is a randomly scaled scale-decorated Poisson point process (SScDPPP).
%Furthermore, for every positive, continuous function $f: \real \to \real$, vanishing in a neighborhood of $0$,
%\aln{
%& \estar \bigg( \exp \Big\{- N_*(f) \Big\}\bigg) \nonumber \\
%&=\estar \Bigg[ \exp \Bigg\{ - \frac{1}{\mu} W \sum_{i=0}^{\infty} \frac{1}{\mu^{i}}  \int_{\real^\infty \setminus \{\mathbf{0}_\infty\}} \exptn \bigg[  \sum_{A \in \pow(\bbn(\tilde{U}_1)) \setminus \{\emptyset\}} \bigg(1- \exp\bigg\{ - \sum_{m \in A} \tilde{Z}_{i}^{(m)} g(x_m) \bigg\} \bigg)  \nonumber \\
%& \hspace{1.5cm} \prob \Big( Z_1 >0 \Big) \Big(\prob(Z_{i}>0) \Big)^{|A|} \Big( \prob(Z_{i} =0) \Big)^{\tilde{U}_1 - |A|}  \bigg] \dtv \lambda(\bld{x})      \Bigg\} \Bigg].  \label{eq:laplace_limit}
%}
\end{thm}

The above result extends Theorem~2.1 of \citet{bhattacharya:hazra:roy:2014} to the case where the displacements of particles coming from the same parent are allowed to be dependent. The proof, however, is much more involved due to presence of a stronger dependence among the displacements coming from the same parent and uses Theorem~\ref{thm:mainthm1} above as one its main ingredients. As a consequence of Theorem~\ref{thm:mainthm2}, we can compute the asymptotic distribution of the position of the rightmost particle in the $n^{th}$ generation, extending Theorem~1 of \citet{durrett:1983} to the dependent displacements case. Qualitatively speaking, the rightmost particle exhibits a similar long run behaviour although its limiting distribution has a scaling constant that depends on the measure $\lambda$, as shown by the following corollary.

\begin{cor} \label{thm:maxima}
Define $M_n = \max_{|\uv|=n} S_\uv$ to be the position of the rightmost particle of the $n^{th}$ generation. Under the assumptions of Theorem \ref{thm:mainthm2}, for every $x>0$,
\aln{
\lim_{n \to \infty} \pstar\big( c_n^\inv M_n \leq x\big) = \estar \Big[ \exp \Big\{- \kappa_\lambda W x^{-\alpha} \Big\}\Big] , \label{eq:maxima_lemma}
}
where $\kappa_\lambda >0$ is a deterministic constant that depends on $\lambda$ and is specified in \eqref{eq:maxima_c_lambda} below.
\end{cor}
\noindent
\subsection{Discussions}
As mentioned earlier, \citet{brunet:derrida:2011} predicted that the limits of point processes of properly normalized positions of particles in branching Brownian motion and branching random walks should be \emph{decorated Poisson point processes} (DPPP) and they should satisfy a \emph{superposability property}. These conjectures were established for branching random walks with light-tailed step sizes by \citet{madaule:2011}, and for branching Brownian motion by \citet{arguin2012poissonian, arguin2013extremal} and \citet{aidekon2013branching}. However, all of these works contained an extra random shift coming from the limit of the underlying derivative martingale. It is expected that superposability will change to stability and DPPP will become ScDPPP as we pass from light-tailed to heavy-tailed displacements. In addition to these, the random shift will now be converted to a random scaling based on the martingale limit $W$ giving rise to the last part of Theorem~\ref{thm:mainthm2}. This part, however, is of a purely existential nature in the sense that the SScDPPP representation cannot be constructed explicitly in most cases. We have been able to compute it only in two very special cases: (i) when the displacements are i.i.d.\ and (ii) when $Z_1$ is a bounded random variable; see Corollaries~\ref{propn:mainresult_bhr} and \ref{cor:bounded_offspring} below.

%\begin{thm} \label{thm:sscdppp}
%The point process $N_*$ in Theorem \ref{thm:mainthm2} admits an SScDPPP representation.
%\end{thm}

We would also like to stress that the role of the derivative martingale is washed away in the heavy-tailed case because, with very high probability, exactly one of the independent copies of the point process $\pointl$ (along with its descendants) survives the scaling by $c_n$. This can be thought of as a ``principle of one big jump'' at the level of point processes (see Lemma \ref{lemma:Limit_Laplace} and Lemma \ref{lemma:treerv}). In the context of branching random walks with heavy-tailed displacements, this principle has been observed for displacements; see \citet{durrett:1979, durrett:1983}, \citet{bhattacharya:hazra:roy:2014} and \citet{maillard:2015}. However, one a big jump principle for point processes is novel and can be used to give a heuristic justification of the limit $N_*$ as described below.

Exactly one of the point processes will survive the normalization and, using a standard argument, it is easy to see that this point process will have progeny up to $L$ $(\sim Geo(1/\mu))$ many generations in the limit. If $\prob(Z_1=0)=0$, then the surviving point process will have $Z_1$ many contributing points because the absence of leaves will force each of its points to go all the way down to the $L^{th}$ generation. These points $\xi_1, \xi_2, \ldots, \xi_{Z_1}$ will repeat $Z^1_L$, $Z^2_L$, \ldots , $Z^{Z_1}_L$ (these are the $T_{lk}$'s in our notation) many times, respectively, where $\{Z^1_n\}, \{Z^2_n\}, \{Z^3_n\}, \ldots $ are independent copies of the underlying Galton Watson process independent of $Z_1$. On the other hand, when $\prob(Z_1=0)>0$, the so-called ``largest point process'' may not contribute at all because all of the trees below it may die. Therefore, one needs to condition on at least one tree below to survive. This means, in particular, that the number of contributing points will become $V$, which has the same distribution as $Z_1$ size-biased by the event that at least one of the $Z_1$ trees below survive up to the $L^{th}$ generation. The conditional distribution of the $T_{lk}$'s given $V$ can now be justified in a similar fashion - they have the same distribution as in the $\prob(Z_1=0)=0$ case, except that we have to condition on the survival of at least one of these $V$ many trees that lie beneath.

\section{Scale-decorated Poisson Point Processes} \label{sec:stablepp}
In this section, we study the stability of point measures and derive equivalent criteria for SScDPPP using Laplace functionals. The study of these equivalent criteria are motivated by the recent investigations of \citet{subag:zeitouni:2014} and also by the work of \citet{davydov:molchanov:zuyev:2008}. Laplace functionals of such point processes become particularly important in analysing the limit arising in the branching random walk. We shall discuss this later. We begin by introducing some notations that will be useful throughout this section.

Let $\ckr$ \label{not:ckr} denote the space of all nonnegative continuous functions defined on $\bar{\mathbb{R}}_0$ with compact support (and hence vanishing in a neighbourhood of $0$). By an abuse of notation, for a measurable function $f: \bar{\mathbb{R}}_0 \to [0, \infty)$, we denote by $\mbfs_{y}f(\cdot)$ the function $f(y \cdot)$. For a point process $N$ on $\bar{\real}_0$ and any $y>0$,  one has $\int f \dtv \mbfs_y N = \int \mbfs_y f \dtv N$. The Laplace functional of a point process $N$ will be denoted by
\aln{
\Psi_N  (f) = \exptn \Big( \exp \Big\{- N(f) \Big\} \Big), \label{eq:lap_not}
}
where $N(f) = \int f \dtv N$. \label{not:int_measure} In parallel to the notion of shifted Laplace functional from \citet{subag:zeitouni:2014}, we define the {\bf scaled Laplace functional} as $\Psi_N(f\|y) := \Psi_N(\mbfs_{y^\inv} f)$  \label{not:scale_laplace} for some $y>0$. We define $[g]_{sc}=\{f \in \ckr : f = \mbfs_y g \mbox{ for some } y>0 \}$ \label{not:g_sc} to be the equivalence class of $g$ under equality of two functions up to scaling. Let us define by $\fralpha(x)$ the Frech\'et distribution function, i.e., for each $\alpha>0$,
$$ \fralpha(x)=\exp(-x^{-\alpha}), \qquad x>0. \label{not:fralpha}$$

\begin{defn}[Scale-uniquely supported]
The scaled Laplace functional of the point process $N$ is uniquely supported on $[g]_{sc}$ if,  for any $f \in \ckr$, there exists a constant $c_f$ (depending on $f$ only) such that $\Psi_N( f\|y )= g(y c_f)$ for all $y>0$.
\end{defn}
The notion of scale-uniquely supported is intimately tied to the behaviour of the scale-decorated Poisson point process. In fact, we show the relation  between the equivalence class of $\fralpha$ and the scaled-Laplace functional of the SScDPPP.  Note that, since in Theorem \ref{thm:mainthm2} for branching random walk the scaling is random, a  study of Poisson processes with random scaling will be needed.
%The first two parts of the following Proposition were investigated in \citet{davydov:molchanov:zuyev:2008} and we extend their result. In fact, we give an independent proof of this result later which follows from a general Theorem (see~Theorem \ref{thm:sub_zet}) containing the similar equivalence relation for the randomly scaled scaled decorated Poisson point process.
The following proposition is the analogue of Theorem 10 in \citet{subag:zeitouni:2014}.
\begin{propn}\label{thm:sub_zet}
Let $N$ be a locally finite point process on $\zrer$ satisfying the following assumptions:
\aln{
\prob(N(\bar{\real}_0) > 0) >0 \mbox{ and  }  \exptn \Big( N(\bar{\real}_0 \setminus (-a,a)) \Big) < \infty \label{eq:ass_sub_zet}
 }
for some $a>0$. Let $g : \real_+ \to \real_+$, be a function. Then the following statements are equivalent:
\let\myenumi\theenumi
\let\mylabelenumi\labelenumi
\renewcommand{\theenumi}{Prop\myenumi}
\renewcommand{\labelenumi}{{\rm (\theenumi)}}
\begin{enumerate}

\item $\Psi_N(f\|\cdot)$ is scale-uniquely supported on $[g]_{sc}$ for all $f\in \ckr$. \label{sub:zet:prop1}

\item $\Psi_N(f\|\cdot)$ is scale-uniquely supported on $[g]_{sc}$  for all $f\in \ckr$ and, for some positive random variable $W$,
\beq \label{eq:SLF}
g(y) = \exptn \Big( \fralpha( ycW^\inv)  \Big),
\eeq
where $\fralpha(x) = \exp\{ x^{-\alpha}\}$ denotes the distribution function of the Frechet-$\alpha$ random variable and $c>0$. \label{sub:zet:prop2}

\item $N \sim SScDPPP( \malpha(\dtv x), \point,W)$ for some point process $\point$ and some positive random variable $W$,  where $m_\alpha(\cdot)$ is described in Example \ref{example:regvarprocess}. \label{sub:zet:prop3}

\end{enumerate}

\end{propn}

 The next result is an immediate corollary of the above proposition. The first two equivalent conditions were also studied in \citet{davydov:molchanov:zuyev:2008}.
\begin{cor} \label{propn:strict_stable}
Assume that $\prob(N(\zrer) >0) >0$. Then the following statements are equivalent:
\begin{enumerate}
\let\myenumi\theenumi
\let\mylabelenumi\labelenumi
\renewcommand{\theenumi}{B\myenumi}
\renewcommand{\labelenumi}{{\rm (\theenumi)}}
\item $N$ is a scale-decorated Poisson point process with Poisson intensity $\nu_\alpha(dx)$, where $\nu_\alpha$ is as defined in~\eqref{eq:nualpha}. \label{item:DM1}

\item $N$ is a strictly $\alpha$-stable point process.\label{item:DM2}

\item The scaled Laplace functional of $N$ is scale-uniquely supported on the class $[\fralpha]_{sc}$.\label{item:DM3}

\end{enumerate}
\end{cor}

If the above point processes are supported on the positive part of the real line, then these results can easily be shown to be equivalent to the ones proved by \citet{subag:zeitouni:2014} via the canonical one to one correspondence between the spaces $\mathscr{M}((0, \infty])$ and $\mathscr{M}((-\infty,\infty])$ given by $\sum {\delta_{a_i}} \leftrightarrow \sum {\delta_{\log{a_i}}}$. In fact, the assumption of monotonicity of $g$ can be dropped from Corollary~3 of the aforementioned reference. When the points in Proposition~\ref{thm:sub_zet} (and Corollary~\ref{propn:strict_stable}) take both positive and negative values, one has to mildly mould the proof given in \citet{subag:zeitouni:2014} to the two-sided setup. This slight moulding is too straightforward to merit detailing.  See \citet{thesis:ayan} for sketches of proofs of Proposition~\ref{thm:sub_zet} and Corollary~\ref{propn:strict_stable} above.

\subsection{Proof of Theorem~\ref{thm:mainthm1}}

To prove this result, we shall start with computing the Laplace functional of the scaled superposition of the point processes $\pointl_i$, which are independent copies of $\pointl$. Let $f\in \ckr$ and let, for some $\delta>0$,  the support of $f$ be contained in the outside  of $B(0,\delta)$. Now, using independence and rearranging, we immediately get that
\begin{align}\label{eq:doa6}
\exptn \Bigg( \exp \bigg\{ - \sum_{i=1}^n \mbfs_{b_n^\inv} \pointl_i(f) \bigg\} \Bigg) &
%=\Bigg[ \exptn \Bigg( \exp \bigg\{- \mbfs_{b_n^\inv} \pointl(f) \bigg\} \Bigg) \Bigg]^n\nonumber\\
=\Bigg[ 1- \frac{1}{n}  \Bigg( \int_{\scrmrz} \bigg(1- \exp\Big\{ - \nu(f) \Big\}\bigg) ~ n  \prob(\mbfs_{b_n^{-1} }\pointl \in \dtv \nu) \Bigg) \Bigg]^n.
\end{align}
Note that the convergence of the Laplace functional is equivalent to the convergence of the integral in~\eqref{eq:doa6}.

%\begin{align}
%&\exptn \Bigg( \exp \bigg\{ - \sum_{i=1}^n \mbfs_{b_n^\inv} L_i(f) \bigg\} \Bigg) \nonumber \\
%&= \Bigg[ \exptn \Bigg( \exp \bigg\{- \mbfs_{b_n^\inv} L(f) \bigg\} \Bigg) \Bigg]^n \nonumber \\
%&= \Bigg[ 1- \frac{n}{n} \bigg[ 1- \exptn \Bigg( \exp \bigg\{- \mbfs_{b_n^\inv} L(f) \bigg\} \Bigg) \bigg] \Bigg]^n \nonumber \\
%&= \Bigg[ 1- \frac{1}{n} n \Bigg( \int \bigg(1- \exp\Big\{ - \nu(f) \Big\}\bigg) \dtv \prob_L(\mbfs_{b_n }\nu) \Bigg) \Bigg]^n \nonumber \\
%&= \Bigg[ 1- \frac{1}{n}  \Bigg( \int \bigg(1- \exp\Big\{ - \nu(f) \Big\}\bigg) n \dtv \prob_L(\mbfs_{b_n }\nu) \Bigg) \Bigg]^n
%\end{align}
%Now it is enough to study the convergence of the integral
%\beq \label{eq:doa7}
%\int \Bigg(1- \exp \bigg\{ -\nu(f) \bigg\} \Bigg) n \prob(\mbfs_{b_n^\inv} L \in \dtv\nu)
%\eeq
 Recall from Subsection~\ref{subsec:reg_var} that the regular variation of $\pointl$ is equivalent to the fact that, for every positive, bounded and continuous function $F$  vanishing outside a neighborhood of $\nullm$,
\beq \label{eq:doa8}
\int_{\scrmrz} F(\nu) n \prob(\mbfs_{b_n^\inv} \pointl \in \dtv\nu) \to \int_{\scrmrz} F(\nu) m^*(\dtv \nu).
\eeq
Here, if we choose $F = (1 - \exp{-\nu(f)})$, then $F$ is positive, bounded and continuous, but it is not immediate  whether it vanishes inside a neighborhood of $\nullm$. To bypass this technicality, we use the fact that $f$ vanishes outside a neighborhood of $0$. Fix an $\epsilon >0$, consider the function $F_\epsilon(\nu) = (1- \exp\{- (\int f \dtv \nu - \epsilon)_+\})$. Then it is clear that this function vanishes outside the ball $B(\nullm, \epsilon)$ under the vague metric and $F_\epsilon(\nu) \downarrow F(\nu)$ as $\epsilon\downarrow 0$. Now, by regular variation of $\pointl$, we get
\beq \label{eq:doa9}
\lim_{\epsilon \to 0} \lim_{n \to \infty} \int_{\scrmrz} F_\epsilon(\nu) n \prob(\mbfs_{b_n^\inv} \pointl \in \dtv \nu) = \lim_{\epsilon \to 0} \int_{\scrmrz} F_\epsilon(\nu) m^*(\dtv \nu) = \int_{\scrmrz} F(\nu) m^*(\dtv \nu).
\eeq
Hence to show that the limit of the integral in \eqref{eq:doa6} is the same as the integral in the right-hand side of \eqref{eq:doa9}, it is sufficient to show that
\beq \label{eq:doa10}
\lim_{\epsilon \to 0} \limsup_{n \to \infty}  \int_{\scrmrz}\Big| F_\epsilon (\nu) -F(\nu) \Big| n \prob(\mbfs_{b-n^\inv} \pointl \in \dtv \nu)  =0.
\eeq
Using $|e^{-x}-e^{-y}|\le |x-y|$, we get that
\begin{align} \label{eq:doa11}
  \int_{\scrmrz} \Big| F_\epsilon (\nu) -F(\nu) \Big| n \prob(\mbfs_{b_n^\inv} \pointl \in \dtv \nu) %\nonumber \\
% & = \int_{\scrmrz} \Bigg| \exp \bigg\{ - \Big( \nu(f) -\epsilon \Big)_+ \bigg\} - \exp\bigg\{ -  \nu(f) \bigg\} \Big| n \prob(\mbfs_{b_n^\inv} \pointl \in \dtv \nu) \nonumber \\
& \le  \int_{\scrmrz} \Big| ( \nu(f) -\epsilon)_+ -  \nu(f) \Big| n \prob(\mbfs_{b_n^\inv} \pointl \in \dtv \nu) \nonumber \\
%& \overset{??}= \int_{\{ \nu : 0 < \int f \dtv \nu < \epsilon\} } |\int f \dtv \nu| n \prob(\mbfs_{b_n^\inv} \pointl \in \dtv \nu) + \int_{\{ \nu : \int f \dtv \nu \ge \epsilon\} } \epsilon n \prob(\mbfs_{b_n^\inv} \pointl \in \dtv \nu)\\
%& \le \epsilon \int_{\{ \nu : 0 < \int f \dtv \nu < \epsilon\} } n \prob(\mbfs_{b_n^\inv} \pointl \in \dtv \nu) \nonumber \\
&= 2 \epsilon n \prob \bigg[  \mbfs_{b_n^\inv} \pointl \in \{ \nu :  \nu(f)>0\}\bigg].
\end{align}
By the choice of $f$, we can get an upper bound for  the right-hand side of \eqref{eq:doa11}, namely,
\beq \label{eq:doa12}
\epsilon n \prob \bigg[ \mbfs_{b_n^\inv} \pointl\in \{\nu : \nu[\delta, \infty) \ge 1\} \bigg].
\eeq
Now it is important to note that $\{\nu : \nu[\delta, \infty) \ge 1\}$ is a closed set in $\scrmrz$ and $\nullm \nin \{\nu : \nu[\delta, \infty) \ge 1\}$. Hence using the portmanteau theorem (Theorem 2.1 in \citet{lindskog:resnick:roy:2014}) for HL-convergence, we get that $\limsup_{n \to \infty} n ~ \prob \bigg[ \mbfs_{b_n^\inv} \pointl \in \{\nu : \nu[\delta, \infty) \ge 1\} \bigg] \to m^*\Big(  \{\nu : \nu[\delta, \infty) \ge 1\} \Big)$. So~\eqref{eq:doa10} follows.
%\begin{align} \label{eq:doa13}
%& \lim_{\epsilon \to 0} \limsup_{n \to \infty} \int_{\scrmrz} \Big| F_\epsilon (\nu) -F(\nu) \Big| n \prob(\mbfs_{b_n^\inv}\pointl\in \dtv \nu) \nonumber \\
%& \le \lim_{\epsilon \to 0} \limsup_{n \to \infty} ~ \epsilon ~ n ~ \prob \bigg[ \mbfs_{b_n^\inv} \pointl \in \{\nu : \nu[\delta, \infty) \ge 1\} \bigg] \nonumber \\
%& \le \lim_{\epsilon \to 0} ~ \epsilon ~ m^*\Big(  \{\nu : \nu[\delta, \infty) \ge 1\} \Big) =0.
%\end{align}
Now finally, using this convergence, we can write down the limiting Laplace functional
\beq \label{eq:doa14}
\lim_{n \to \infty } \exptn \bigg[ \exp \bigg\{ - \sum_{i=1}^n \mbfs_{b_n^\inv} \pointl_i(f) \bigg\} \bigg] = \exp \bigg\{ - \int_{\scrmrz} \Big(1- \exp\{-\nu(f)\} \Big) m^*(\dtv \nu) \bigg\}.
\eeq
It turns out that the scaled Laplace functional $m^*$ is scale-uniquely supported on $[\fralpha]$ and hence by Proposition~\ref{propn:strict_stable} (\eqref{item:DM3} implies \eqref{item:DM2}) it follows that the limit is a strictly $\alpha$-stable point process. Indeed, from the above convergence we have
\begin{equation}\label{eq:doa15}
\Psi_{m^*}(f\|y)= \exp \bigg\{ - y^{-\alpha} \int_{\scrmrz} \Big(1- \exp\{-  \nu(f)\} \Big) m^*(\dtv \nu) \bigg\}=\fralpha(yc_f)
\end{equation}
where
$$c_f^{-\alpha}=\int_{\scrmrz} \Big(1- \exp\{-  \nu(f)\} \Big) m^*(\dtv \nu). $$

\section{Computation of the Weak Limit of $N_n$} \label{sec:brw}

Recall from Subsection~\ref{subsec:brw} that, for $\uv\in \mathbb V$, we denote by $\Iv$ the unique geodesic path from the root to $\uv$. We first introduce the point process $\tilde N_n$ that takes into account the one large jump along a typical path $\Iv$. More precisely, we define
\beq \label{eq:olj1}
\tilde{N}_n = \sum_{|\uv| =n} \sum_{\uu \in \Iv} \delta_{c_n^\inv \Xu}.
\eeq
First we state the following important Lemma about the convergence of the point process $\tilde N_n$.
\begin{lemma}\label{lemma:one}
Under the assumptions of Theorem~\ref{thm:mainthm2}, we have under $\pstar$ that $\tilde N_n$ weakly converge to $N_*$.
\end{lemma}
Lemma~\ref{lemma:one}  immediately implies the result,  since by retracing the proof of Lemma 3.1  of \citet{bhattacharya:hazra:roy:2014} we can easily show that
\beq \label{eq:olj2}
\limsup_{n \to \infty} \pstar ( \rho(N_n, \tilde{N}_n) > \epsilon) =0, \text{  for every $\epsilon >0$
},
\eeq
where $\rho$ is the vague metric on $\scrmrz$.  In the rest of this section, we concentrate on the proof of Lemma~\ref{lemma:one}. This can be split into three major steps - {\it cutting, pruning, regularisation}. The first two steps are exactly the same as in the proof of the main theorem of \citet{bhattacharya:hazra:roy:2014} and hence we only sketch them in Subsection~\ref{subsec:cut_prune}. The third step, however, is new and forms the key  towards computation of the weak limit of $\tilde{N}_n$. Subsection~\ref{subsec:regu_weak} contains the details of this step, which uses Theorem~\ref{thm:mainthm1} as one its main ingredients.

\subsection{Cutting and Pruning} \label{subsec:cut_prune}
In order to compute the weak limit of $\tilde{N}_n$, we follow the two-step truncation introduced in \citet{bhattacharya:hazra:roy:2014} for the proof of Theorem 2.1 therein. We briefly sketch these two steps in this subsection and recall the corresponding notations from the aforementioned reference. Let us denote by $D_j$ the set of vertices in the $j^{th}$  generation of the tree $\bbt$. \label{not:D_n}  Fix a positive integer $K$, and choose $n$ large enough such that $n>K$ and cut the tree at the $(n-K)^{th}$ generation to keep the forest containing $K$ generations of $|D_{n-K}|$ independent Galton-Watson trees (which we denote by $\{\bbt_j: 1\le j\le |D_{n-K}|\}$). Each vertex $\uv$ in the $n^{th}$ generation of the original tree belongs to the $K^{th}$ generation of some sub-tree $\bbt_j$. So, given a $\uv\in D_n$, there exists an unique geodesic path from the root of $\bbt_j$ to $\uv$. We denote this path by $\Iv^K$.
 %\begin{equation} \label{eq:ct1}
%\tnk := \sum_{|\uv|=n} \sum_{\uu \in \Iv^K} \delta_{c_n^{-1}\Xu},
%\end{equation}
%where $|\uv|$ denotes the generation of $|\uv|$ in the original tree $\bbt$. Following arguments similar to Lemma 3.2 of \citet{bhattacharya:hazra:roy:2014} and using the definition of regular variation one can easily establish that \beq \label{eq:ct2}
%\lim_{K \to \infty} \limsup_{n \to \infty} \pstar ( \rho(\tilde{N}_n , \tilde{N}_n^K) > \epsilon) =0, \text{ for every $\epsilon >0$}.
%\eeq
%So now it becomes enough to study the convergence of $\tilde N_n^K$.

We prune the forest obtained above. Fix an integer $K>0$ and for each edge $\uu$ in the forest $\cup_{j=1}^{|D_{n-K}|} \bbt_j$, define $A_\uu$ to be the number of descendants of $\uu$ at the $n^{th}$ generation of $\bbt$. Fix another integer $B>1$ large enough so that $\mu_B := \exptn (Z_1^{(B)}) > 1$, where $Z_1^{(B)} := Z_1 \mathbbm{1}(Z_1 \le B) + B \mathbbm{1}(Z_1 > B)$. We modify the forest according to the pruning algorithm introduced in \citet{bhattacharya:hazra:roy:2014} as follows. If the root of $\bbt_1$ has more than $B$ offsprings (edges), then keep the first $B$  of them and remove the others and their descendants. If the number of offsprings of the root is less than or equal to $B$, then do nothing. Repeat this for offsprings of each of the remaining vertices. Continue this up to the offsprings of the $(K-1)^{th}$ generation vertices in $\bbt_1$ to obtain its $B$-pruned version $\bbt^{(B)}_1$. Similarly, apply the same procedure, to get $\bbt^{(B)}_j$, the pruned version of $\bbt_j$ for each $j \leq D_{n-K}$.

Note that under $\prob$, these $|D_{n-K}|$ pruned sub-trees are independent copies of a Galton-Watson tree (up to the $K^{th}$ generation) with a bounded branching random variable $Z_1^{(B)}$. For each edge $\uu$ in $\cup_{j=1}^{|D_{n-K}|} \bbt^{(B)}_j$, we define $\aub$ to be the number of descendants of $\uu$ in the $K^{th}$ generation of the corresponding pruned sub-tree. Observe that for every vertex $\uu$ at the $i^{th}$ generation of any sub-tree $\bbt^{(B)}_j$, $\aub$ is equal in distribution to $Z_{K-i}^{(B)}$, where $\{ Z_i^{(B)} \}_{i \ge 0}$ denotes a branching process with $Z_0^{(B)} \equiv 1$ and branching random variable $Z_1^{(B)}$. For each $i = 1, 2. \ldots, K$, we denote by $D_{n-K+i}^{(B)}$ the union of all $i^{th}$ generation vertices (as well as edges) from the pruned sub-trees $\bbt^{(B)}_j$, $j=1,2,\ldots, |D_{n-K}|$.

 We introduce the following useful point processes:
 \begin{equation} \label{eq:ct1}
\tnk := \sum_{|\uv|=n} \sum_{\uu \in \Iv^K} \delta_{c_n^{-1}\Xu}, \qquad \text{and } \qquad \nkb := \sum_{\uv \in D_n^{(B)}} \sum_{\uu \in \Iv^K} \delta_{b_n^{-1} \Xu},
\end{equation}
where $|\uv|$ denotes the generation of $|\uv|$ in the original tree $\bbt$.
%\begin{equation} \label{eq:ptf1}
%\nkb := \sum_{\uv \in D_n^{(B)}} \sum_{\uu \in \Iv^K} \delta_{b_n^{-1} \Xu}.
%\end{equation}
The point processes $\tnk$ and $\nkb$ are not simple point processes since both of them have alternative representations:
\begin{equation} \label{eq:ptf2}
\tnk = \sum_{i=0}^{K-1} \sum_{\uu \in D_{n-i}} A_\uu \delta_{b_n^{-1} \Xu}, \qquad \text{and } \qquad \nkb = \sum_{i=0}^{K-1} \sum_{\uu \in D_{n-i}^{(B)}} \aub \delta_{b_n^{-1} \Xu}.
\end{equation}
%and
%\begin{equation} \label{eq:ptf3}
%\nkb = \sum_{i=0}^{K-1} \sum_{\uu \in D_{n-i}^{(B)}} \aub \delta_{b_n^{-1} \Xu}.
%\end{equation}
%It follows from the arguments immediately in \cite[Lemma 3.3 ]{bhattacharya:hazra:roy:2014} that $\lim_{B \to  \infty} \limsup_{n \to \infty} \pstar ( \rho(\tilde{N}_n^K, \tilde{N}_n^{(K,B)}) > \epsilon) =0$ for every $\epsilon>0$. So this reduces our work to computation of weak limit of \eqref{eq:ptf3} obtained by letting $n \to \infty$, and then $B \to \infty$, and finally $K \to \infty$.

%\begin{lemma} \label{lemma3}
%  Under the assumptions of Theorem~\ref{mainthm}, for each fixed positive integer $K>1$ and for all $\epsilon > 0$,
%\begin{equation} \label{pmr10}
%\lim_{B \to \infty} \limsup_{n \to \infty} \pstar \lfb \rho \lfb \tnk, \nkb \rfb >\epsilon \rfb=0.
%\end{equation}
%\end{lemma}

%\begin{lemma} \label{lemma:prunning_forest}
%Under the assumptions,
%\beq \label{eq:ptf4}
%\lim_{B \to  \infty} \limsup_{n \to \infty} \pstar ( \rho(\tilde{N}_n^K, \tilde{N}_n^{(K,B)}) > \epsilon) =0
%\eeq
%for every $\epsilon >0$.
%\end{lemma}

The following lemma summarises the reason why the investigation of the weak convergence of $\nkb$ is enough to prove Lemma~\ref{lemma:one}. Since this lemma can be derived by appropriate modifications of the proofs of Lemmas~3.2 and 3.3 of \citet{bhattacharya:hazra:roy:2014}, using the definition of regular variation, we skip its proof in this paper. For details, the readers are referred to \citet{thesis:ayan}.
\begin{lemma}
Under the assumptions of Theorem~\ref{thm:mainthm2} it follows that:
\begin{enumerate}
\item For every $\epsilon >0$
\beq \label{eq:ct2}
\lim_{K \to \infty} \limsup_{n \to \infty} \pstar ( \rho(\tilde{N}_n , \tilde{N}_n^K) > \epsilon) =0.
\eeq
\item For every positive integer $K$ and every $\epsilon > 0$,
\begin{equation} \label{pmr10}
\lim_{B \to \infty} \limsup_{n \to \infty} \pstar \left( \rho ( \tnk, \nkb )>\epsilon\right)=0.
\end{equation}
\end{enumerate}

\end{lemma}

\subsection{Regularization of the Pruned Forest} \label{subsec:regu_weak}
The study of the weak convergence of $\nkb$ and the identification of the limit is the main technically challenging step, which is carried out through the lemma presented below.

 \begin{lemma} \label{lemma:Limit_Laplace}
Under the assumptions of Theorem~\ref{thm:mainthm2}, for all $K \geq 1$ and for all $B$ large enough so that $\mu_B = \exptn (Z_1^{(B)}) > 1$, there exists point processes $N_*^{(K,B)}$ and $N_*^{(K)}$ such that, under $ \pstar$,
\begin{itemize}
\item[(a)]  $\nkb\Rightarrow N_*^{(K,B)}$ as $n\to \infty$,
\item[(b)]$N_*^{(K,B)}\Rightarrow N_*^{(K)}$  as $B\to\infty$,  
\item[(c)] $N_*^{(K)}\Rightarrow N_*$ as  $K\to\infty$,
\end{itemize}
%as $n \to \infty$,   $B \to \infty$ and $K \to \infty$
%\aln{
%N_n'^{(K,B)} \Rightarrow N_*
%}
in the space $\scrm(\bar{\real}_0)$ equipped with the vague topology. Furthermore, $N_*$ admits the representation \eqref{eq:rep_limit_measure} and also an SScDPPP representation.
\end{lemma}

\vspace{.5cm}

\begin{center}

\begin{tikzpicture}

% generation 2

\draw[dashed] (-5.5,-2) -- (5,-2);

 \draw[ ->] (-2,-2) --  (-5,-4) ;

 \draw[ ->] (2,-2) --  (2.5,-4);
 \draw[ ->] (2,-2) --  (5,-4);

 %generation 3

 \draw[ ->] (-5,-4) -- (-5.5,-6);
 \draw[ ->] (-5,-4) -- (-4.5,-6);

 \draw[ ->] (2.5,-4) -- (2.5,-6);

 \draw[ ->] (5,-4) -- (5,-6);
 \draw[ ->] (5,-4) -- (5.5,-6);

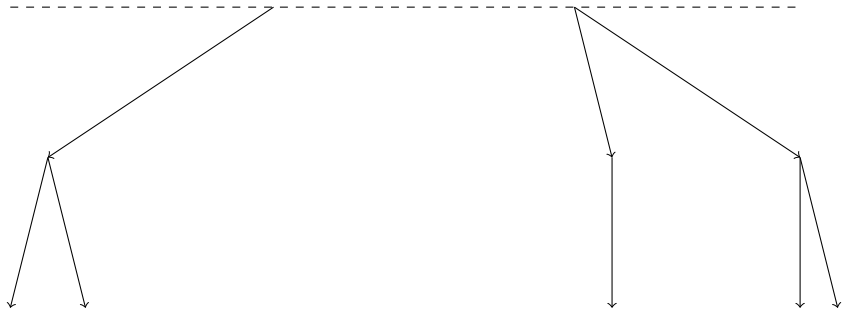
 \draw (0,-6.5) node[anchor=north]{ \small{Figure 1  : Afer cutting and pruning with $K=2$ and $B=2$.}};

 \end{tikzpicture}

\end{center}

We shall use an idea of regularisation to derive the lemma. After pruning the trees $\{\bbt_j^{(B)} : j \ge 1\}$, we shall make them a bunch of regular subtrees following the algorithm: (see Figure 1)

\begin{enumerate}

\item[R1.] Fix a subtree $\bbt_1^{(B)}$ and look at its root.

\item[R2.] The root can have at most $B$ children. If it has exactly $B$  children, then do nothing. Otherwise, if it has $m < B$ children, then add $B-m$ new vertices. Define $\aub :=0$ if $\uu$ is a newly added vertex.

\item[R3.] Now we have exactly $B$  particles at the first generation of the subtree $\bbt_1^{(B)}$ and the next step is to replace their displacements by an independent copy of $(X_1, X_2, \ldots, X_B)$.

\item[R4.] Follow the steps R2 and R3 for each of the $B$-members of the first generation and continue this up to the $K^{th}$ generation.

\item[R5.] Repeat the steps R1, R2, and R3 for each of the other subtrees.

\end{enumerate}

\noindent See Figure~2  below for the regularized versions of the pruned subtrees (as in Figure~1 above). The newly added edges are the dotted ones.

\vspace{.75cm}

\begin{center}

\begin{tikzpicture}

% generation 2

\draw[dashed] (-5.5,-2) -- (5,-2);

 \draw[ ->] (-2,-2) --  (-5,-4) ;
 \draw[dotted, ->] (-2,-2) -- (-1,-4);

 \draw[ ->] (2,-2) --  (2.5,-4);
 \draw[ ->] (2,-2) --  (5,-4);

 %generation 3

 \draw[ ->] (-5,-4) -- (-5.5,-6);
 \draw[ ->] (-5,-4) -- (-4.5,-6);

\draw[dotted,->] (-1,-4) -- (0,-6);
\draw[dotted,->] (-1,-4) -- (-2,-6);

 \draw[ ->] (2.5,-4) -- (4,-6);
 \draw[dotted,->] (2.5,-4) -- (1.5,-6);

 \draw[ ->] (5,-4) -- (5,-6);
 \draw[ ->] (5,-4) -- (5.5,-6);

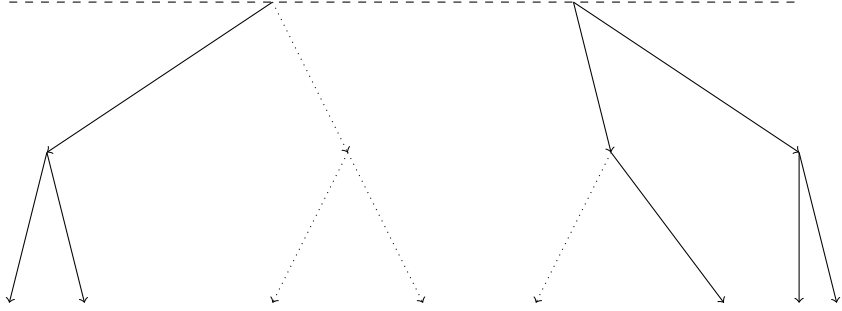
 \draw (0,-6.5) node[anchor=north]{ \small{Figure 2  : After regularisation with {K=2} and $B=2$.}};

 \end{tikzpicture}

\end{center}

It is important to note that the displacements corresponding to the subtrees are changed but have the same distribution. After modification, the modified displacement corresponding to the vertex $\uu$ will be denoted by $\Xu'$. So, we have a new point process
\aln{
N_n'^{ (K,B)} := \sum_{i=0}^{K-1} \sum_{\uu \in D_{n-i}^{(B)}} \aub \delta_{b_n^\inv X'_\uu},
}
 which has the same distribution as $\nkb$. Here we shall use the idea that the point process corresponding to the subtrees are independently and identically distributed and that $\dnkb$ is the superposition of the point processes corresponding to the subtrees. We shall show that the point process corresponding to a fixed subtree is regularly varying in the space $\scrmrz$.

After employing the regularization algorithm, we denote the modified trees by $\{\tilde{\bbt}_j^{(B)} : j \ge 1 \}$. We denote   $l^{th}$ vertex at the $i^{th}$ generation of the $j^{th}$ subtree by the triplet $(j,i,l)$. Then we observe that
\beq \label{eq:polf1}
N_n'^{ (K,B,j)}=\sum_{i=1}^K \sum_{l=1}^{B^i} A_{(j,i,l)}^{(B)} \delta_{ X'_{(j,i,l)}} \qquad\text{and }\qquad \dnkb = \sum_{j=1}^{|D_{n-K}|} N_n'^{ (K,B,j)}.
\eeq
% It is enough to study the limit of the Laplace functional of $\dnkb$.
 The first step is to show that the point process $\dnkbo$  is regularly varying in the space $\scrmrz$. The following lemma is the backbone of the proof of Lemma~\ref{lemma:Limit_Laplace}.
 \begin{lemma}\label{lemma:treerv}
$\dnkb$ is the superposition of $|D_{n-K}|$ independent and identically distributed copies of the point process $\dnkbo$ (each of which is independent of $|D_{n-K}|$) and there exists a non-null measure $\Upsilon$ on $\scrmrz$ such that
\aln{
 \Upsilon_n:=\mu^n \prob(\mbfs_{c_n^\inv} \dnkbo \in \cdot) \hlconv \Upsilon,
 \label{eq:regvar_subtree}
}
where $\Upsilon(B_r) < \infty$ for every $r >0$ with
\aln{
B_r = \{ \nu \in \scrmrz: \rho(\nu, \nullm) >r\}. \label{eq:defn_Br}
}
\end{lemma}
\subsection{Regular Variation of $\dnkbo$: Proof of Lemma \ref{lemma:treerv}}
 %\begin{proof}%[Proof of Lemma~\ref{lemma:treerv}]
We start with some preliminary notations and important observations. Define
\beq \label{eq:polf2}
\tilde{A}_j=(A_{(j,1,1)}^{(B)}, \ldots , A_{(j,1,B)}^{(B)},\ldots ,A_{(j,K,1)}^{(B)}, \ldots , A_{(j,K,B^K)}^{(B)})
\eeq
and
\beq \label{eq:polf3}
\tilde{X}_j =(X'_{(j,1,1)}, \ldots, X'_{(j,1,B)}, \ldots , X'_{(j,K,1)}, \ldots , X'_{(j,K,B^K)}).
\eeq
Here, $\tilde{A}_j$ denotes the collection $\{\aub\}$ for the $j$-th regularized tree, which  is an element of $\tilde{S}_B$ with common law $G(\cdot)$, where
\beq \label{eq:polf4}
\tilde{S}_B =  \underbrace{S_{B^{K-1}} \times \ldots \times  S_{B^{K-1}}}_{B~ many} \times \ldots \times \underbrace{S_B \times \ldots \times S_B}_{B^K ~ many}
\eeq
and $S_p =\{0,1, \ldots, p\}$, while $\tilde{X}_j$ denotes the collection $\{\Xu'\}$ for the $j$-th regularized tree and is an element of
\beq \label{eq:polf5}
\tilde{R}_B=\underbrace{ \real \times \ldots \times \real}_{B ~ many} \times \ldots \times \underbrace{\real \times \ldots \times \real}_{B^K ~ many}.
\eeq
By  construction $\{\tilde{X}_j : j \ge 1\}$ is an i.i.d. collection of $\tilde{R}_B$-valued random elements and also independent of the collection $\{\tilde{A}_j : j \ge 1\}$, which are also i.i.d. themselves. It is important to note that the  convergence in \eqref{eq:jtregvar_cn} implies
\aln{
\mu^n \prob( c_n^\inv (X_1, \ldots, X_B) \in \cdot) \hlconv \lambda^{(B)}(\cdot) \label{eq:polf6}
}
on the space $\real^{B} \setminus \{\mathbf{0}_B\}$, \label{not:zero_b} where $\mathbf{0}_B = (0, 0, \ldots, 0) \in \real^B$ and $\lambda^{(B)} = \lambda \circ \mbox{PROJ}_B^\inv $, with $\mbox{PROJ}_B$  an operator on $\real^\bbn$ such that $\mbox{PROJ}_B ((u_i)_{i=1}^\infty) = (u_1, \ldots, u_B)$ ( Theorem~4.1 of \citet{lindskog:resnick:roy:2014}). Using \eqref{eq:polf6}, it is easy to see that
\aln{
|D_{n-K}| \prob( c_n^\inv \tilde{X}_1 \in \cdot) = \frac{1}{\mu^K} \frac{Z_{n-K}}{\mu^{n-K}} \mu^n \prob(c_n^\inv \tilde{X}_1 \in \cdot) \hlconv \frac{1}{\mu^K} W \tau(\cdot) \label{eq:polf7}
}
$\prob$-almost surely on $\tilde{R}_B \setminus \{\mathbf{0}\}$, where $\mathbf{0} \in \tilde{R}_B$ with all its components $0$ and
\aln{
\tau := \sum_{i=1}^K \sum_{l \in J_i} \tau_{i,l} := \sum_{i=1}^K \sum_{l \in J_i} \bigotimes_{t=1}^{B+B^2 + \ldots + B^{i-1} + l-1} \delta_0 \otimes \lambda^{(B)} \otimes \bigotimes_{t= B+ \ldots + B^{i-1} +B}^{B+B^2 + \ldots + B^K} \delta_0, \label{eq:polf8}
}
$W$ is the martingale limit associated to the branching process (see \eqref{eq:martngle_conv}) and $J_i = \{ p \in \{ 1, \ldots, B^i\} : p \equiv 1 \mbox{ mod } B\}$. Now combining the above result with the fact that $\tilde{A}_1 $ and $\tilde{X}_1$ are independent, we get
\aln{
|D_{n-K}| \prob( \tilde{A}_1 \in \cdot, c_n^\inv \tilde{X}_1 \in \cdot) \hlconv \frac{1}{\mu^K} W \otimes G(\cdot) \otimes \tau(\cdot) \label{eq:polf9}
}
$\prob$-almost surely on $\tilde{S}_B \times (\tilde{R}_B \setminus \{ \mathbf{0}\})$.

In order to show that $\Upsilon_n\hlconv \Upsilon$, we shall use \citet[Theorem A.2]{hult:samorodnitsky:2010}. Fix $g_1, g_2\in \ckr$ (with $\mathrm{support}(g_i) \subseteq\{x : |x| > \eta_i\}$ for $i=1,2$) and $\epsilon_1,\epsilon_2>0$, and define a map
$F_{g_1,g_2,\eps_1,\eps_2}: \scrm(\bar{\real}_0)\to [0,\infty)$ by
\begin{equation*}
F_{g_1,g_2,\eps_1,\eps_2}(\nu)=\left( 1-\exp(-(\nu(g_1)-\epsilon_1)_{+})\right)\left( 1-\exp(-(\nu(g_2)-\epsilon_2)_{+})\right).
\end{equation*}
By the aforementioned result, to establish $\Upsilon_n\hlconv \Upsilon$, we have to verify that
$$\Upsilon_n( F_{g_1,g_2,\eps_1,\eps_2})\to \Upsilon(F_{g_1,g_2,\eps_1,\eps_2})$$
 as $n \to \infty$. From the earlier observations we have that,
 \begin{align} \label{eq:apll1}
 \Upsilon_n( F_{g_1,g_2,\eps_1,\eps_2})&=  \sum_{\tilde{a} \in \tilde{S}_B} \int_{\tilde{R}_B} \bigg[ \bigg( 1 - \exp \bigg\{ - \Big( \sum_{i=1}^K \sum_{l=1}^{B^i} a_{i,l} g_1(c_n^\inv x_{i,l}) - \epsilon_1 \Big)_+\bigg\} \bigg) \nonumber\\
 &\bigg( 1 - \exp \bigg\{ - \Big( \sum_{i=1}^K \sum_{l=1}^{B^i} a_{i,l} g_2(c_n^\inv x_{i,l}) - \epsilon_2 \Big)_+\bigg\}\bigg) \bigg] \mu^n \prob(c_n^\inv \tilde{X}_1 \in \dtv \tilde{x}) G(\tilde{a})
 \end{align}
% We shall compute the limit of
%\begin{align} \label{eq:apll1}
%&\mu^n \exptn \bigg[ \bigg( 1 - \exp \bigg\{ - \Big( \mbfs_{c_n^\inv}\dnkbo(g_1) - \epsilon_1 \Big)_+\bigg\} \bigg) \bigg( 1 - \exp \bigg\{ - \Big( \mbfs_{c_n^\inv}\dnkbo(g_2) - \epsilon_2 \Big)_+\bigg\}\bigg) \bigg] \nonumber \\
%&= \mu^n \exptn \bigg[ \bigg( 1 - \exp \bigg\{ - \Big( \sum_{i=1}^K \sum_{l=1}^{B^i} A^{(B)}_{i,l}g_1(c_n^\inv X'_{i,l}) - \epsilon_1 \Big)_+\bigg\} \bigg) \nonumber \\
%& \hspace{4cm} \bigg( 1 - \exp \bigg\{ - \Big( \sum_{i=1}^K \sum_{l=1}^{B^i} A^{(B)}_{i,l}g_2(c_n^\inv X'_{i,l}) - \epsilon_2 \Big)_+\bigg\}\bigg) \bigg] \nonumber \\
%& =  \sum_{\tilde{a} \in \tilde{S}_B} \int_{\tilde{R}_B} \bigg[ \bigg( 1 - \exp \bigg\{ - \Big( \sum_{i=1}^K \sum_{l=1}^{B^i} a_{i,l} g_1(c_n^\inv x_{i,l}) - \epsilon_1 \Big)_+\bigg\} \bigg) \nonumber \\
%&  \hspace{1cm} \bigg( 1 - \exp \bigg\{ - \Big( \sum_{i=1}^K \sum_{l=1}^{B^i} a_{i,l} g_2(c_n^\inv x_{i,l}) - \epsilon_2 \Big)_+\bigg\}\bigg) \bigg] \mu^n \prob(c_n^\inv \tilde{X}_1 \in \dtv \tilde{x}) G(\tilde{a})
%\end{align}
as $n \to \infty$.
Now consider the function $h(., \tilde a): \tilde{R}_B \to \real_+$ such that
\beq \label{eq:apll2}
h(\tilde{x}, \tilde a) = \bigg( 1 - \exp \bigg\{ - \Big( \sum_{i=1}^K \sum_{l=1}^{B^i} a_{i,l} g_1(c_n^\inv x_{i,l}) - \epsilon_1 \Big)_+\bigg\} \bigg) \bigg( 1 - \exp \bigg\{ - \Big( \sum_{i=1}^K \sum_{l=1}^{B^i} a_{i,l} g_2(c_n^\inv x_{i,l}) - \epsilon_2 \Big)_+\bigg\}\bigg).
\eeq
It is easy to see that the integrand in \eqref{eq:apll1} is a bounded and continuous function on $\tilde{R}^B$ that vanishes in a neighbourhood of $\bld{0} \in \tilde{R}_B$. Using the convergence stated in \eqref{eq:polf9} and \eqref{eq:polf8}, we get that the right-hand side of  \eqref{eq:apll1} converges to
\begin{align} \label{eq:apll3}
& \sum_{\tilde{a} \in \tilde{S}_B} \int_{\tilde{R}_B}h(\tilde x,\tilde a)  \tau(\dtv \tilde{x}) G(\tilde{a})
= \sum_{\tilde{a} \in \tilde{S}_B} \sum_{i=1}^K\sum_{l\in J_i} \int_{\tilde{R}_B} h(\tilde x, \tilde a)\tau_{i,l}(\dtv \tilde{x}) G(\tilde{a})  \nonumber \\
&= \sum_{\tilde{a} \in \tilde{S}_B} \sum_{i=1}^K\sum_{l\in J_i} \int_{\real^B} \bigg[ \bigg( 1 - \exp \bigg\{ - \Big( \sum_{t=l}^{l+B-1} a_{i,t} g_1( x_{t-l+1}) - \epsilon_1 \Big)_+\bigg\} \bigg) \nonumber \\
& \hspace{1cm} \bigg( 1 - \exp \bigg\{ - \Big(  \sum_{t=l}^{l+B-1} a_{i,t} g_2( x_{t-l+1}) - \epsilon_2 \Big)_+\bigg\}\bigg) \bigg] \lambda^{(B)}(\dtv \tilde{x}) G(\tilde{a})
\end{align}
as $n \to \infty$ where $\tau(\cdot)$ is described in \eqref{eq:polf8}.

To compute the above integral and give a probabilistic interpretation of the integral, let us define a collection of independent random variables as follows
\begin{itemize}
\item $\{Y_i^{(B)} : 1 \le i \le K\}$ is a collection of independent random variables such that $Y_i^{(B)} \eqd Z_i^{(B)}$ for every $1 \le i \le K$.
\item For each $1\le i\le K$,  $\{Z_i^{(m,B)} : 1 \le m \le B \}$ is a collection of independent copies of $Z_i^{(B)}$.
\item $\{U_j^{(B)} : j \ge 1\}$ is such that $U_j^{(B)}$ are independent copies of random variable $Z_1^{(B)}$.
\end{itemize}

%Let $\{Y_i^{(B)} : 1 \le i \le K\}$ be a collection of independent random variables such that $Y_i^{(B)} \eqd Z_i^{(B)}$ for every $1 \le i \le K$. Let $\{ Z_i^{(m,B)} : 0 \le i \le K-1, 1 \le m \le B \}$ be a collection of independent random variables such that $\{Z_i^{(m,B)} : 1 \le m \le B \}$ is the collection of independent copies of $Z_i^{(B)}$ for every fixed $i$. We assume that $\{Y_i^{(B)} : 1 \le i \le K\}$ and $\{ Z_i^{(m,B)} : 0 \le i \le K-1, 1 \le m \le B \}$ are independent collections of random variables. We also consider another collection of random variables $\{U_j^{(B)} : j \ge 1\}$ such that $U_j^{(B)}$ are independent copies of random variable $Z_1^{(B)}$ and independent of $\{ Z_i^{(m,B)} : 0 \le i \le K-1, 1 \le m \le B \}$ and $\{Y_i^{(B)} : 1 \le i \le K\}$.

Let $\tilde{U}_j^{(B)}$ denote the random variable $U_j^{(B)}$ conditioned to stay positive for $j \ge 1$ and $\tilde{Z}_{i}^{(m,B)}$ denotes the random variable $\tilde{Z}_i^{(m,B)}$ conditioned to stay positive for every $0 \le i \le (K-1) $ and $1 \le m \le B$.

First fix a generation $1\le i\le K$. In order to compute the expectation of the exponent in \eqref{eq:apll3} with respect to the law $G(\cdot)$, we need to consider only those members of the $i^{th}$ generation that have at least one descendant at the $K^{th}$ generation of $\tilde{\bbt}_1^{(B)}$.
%Here we see that in the exponent the members of $i^{th}$ generation are clubbed if they have same parent as the displacements are dependent.
So we start with $Y_{i-1}^{(B)}$ particles at the $(i-1)^{th}$ generation (potential parents of particles at the $i^{th}$ generation). Each of the members of the $(i-1)^{th}$ generation has a random number of children distributed as $Z_1^{(B)}$ being independent of others. We consider those particles of the $(i-1)^{th}$ generation that have at least one child at the $i^{th}$ generation.  Say the $j^{th}$ member of the $(i-1)^{th}$ generation has $U_j^{(B)}$ children at the $i^{th}$ generation such that $U_j^{(B)} \ge 1$. Among these children only those will be considered who have at least one descendant at the $K^{th}$ generation. This means that we choose a subset of children such that each child has at least one descendant at the  $K^{th}$ generation. So, the random number of children corresponding to each of the chosen particles at the $(i-1)^{th}$ generation has the same distribution as $Z_1^{(B)}$ conditioned to stay positive, being independent of others. Each of the particles chosen from the $i^{th}$ generation must have the same distribution as that of $Z_{K-i}^{(B)}$ conditioned to stay positive, being independent of others.

 We now introduce some new notations that will be essential for the computation. Let $[p] = \{1, 2, \ldots, p\}$. By $\pow (A)$ we mean the collection of all possible subsets of $A$, including the null-set and the set itself. From the above discussion it is clear that we shall choose a random subset of $[\tilde{U}_j^{(B)}]$ such that the elements of the chosen subset have at least one descendant at the $K^{th}$ generation with law the same as $Z_{K-i}^{(B)}$ conditioned to be positive.
 % The randomly chosen subset of $\bbn(p)$ will be denoted by $\eff(\bbn(p))$. It is easy to see that
%\aln{
%\prob(\eff(\bbn(\tilde{U}_j^{(B)})) = A) = \sum_{k= |A|} \prob(\tilde{U}_j^{(B)} = k) \Big[ \prob(Z_{K-i}^{(B)} >0) \Big]^{|A|} \Big[ \prob( Z_{K-i}^{(B)} =0) \Big]^{k- |A|} \label{eq:apll4}
%}
%where $|A|$ denotes cardinality of the subset $A$.

 We shall compute the expectation with respect to $G(\cdot)$ in \eqref{eq:apll3}. To ease the presentation, we define $\eff_B(A)= \prob(U_j^{(B)} >0 ) \prob \Big( Z_{K-i}^{(B)} > 0 \Big)^{|A|} \Big( \prob(Z_{K-i}^{(B)}=0)  \Big)^{\tilde{U}_j^{(B)} - |A|}$. Then,
\begin{align} \label{eq:apll5}
& \sum_{\tilde{a} \in \tilde{S}_B}\sum_{\stackrel{l\in J_i}{ (a_{i,l}, a_{i,l+1}, \ldots, a_{i,l+B-1}) \neq (0, \ldots,0)}} \int_{\real^B} \bigg[ \bigg( 1 - \exp \bigg\{ - \Big( \sum_{t=l}^{l+B-1} a_{i,t} g_1( x_{t-l+1}) - \epsilon_1 \Big)_+\bigg\} \bigg) \nonumber \\
& \hspace{2cm} \bigg( 1 - \exp \bigg\{ - \Big(  \sum_{t=l}^{l+B-1} a_{i,t} g_2( x_{t-l+1}) - \epsilon_2 \Big)_+\bigg\}\bigg) \bigg]  G(\tilde{a}) \nonumber \\
&= \exptn \bigg[ \sum_{j=1}^{Y_i^{(B)}} \sum_{A \in \pow([\tilde{U}_j^{(B)}]) \setminus \{\nullm\}} \bigg( 1- \exp \bigg\{ - \Big(\sum_{m \in A} \tilde{Z}_{K-i}^{(m,B)} g_1(x_m) - \epsilon_1 \Big)_+ \bigg\} \bigg) \nonumber \\
& \hspace{2cm} \bigg( 1- \exp \bigg\{ - \Big(\sum_{m \in A} \tilde{Z}_{K-i}^{(m,B)} g_2(x_m) - \epsilon_2 \Big)_+ \bigg\} \bigg) \eff_B(A)\bigg]\nonumber\\
%\nonumber \\
%& \hspace{3cm} \prob(U_j^{(B)} >0 ) \prob \Big( Z_{K-i}^{(B)} > 0 \Big)^{|A|} \Big( \prob(Z_{K-i}^{(B)}=0)  \Big)^{\tilde{U}_j^{(B)} - |A|} \bigg] \nonumber \\
& = \mu_B^{i-1}  \exptn \bigg[  \sum_{A \in \pow([\tilde{U}_1^{(B)}])\setminus \{\nullm\}} \bigg( 1- \exp \bigg\{ - \Big(\sum_{m \in A} \tilde{Z}_{K-i}^{(m,B)} g_1(x_m) - \epsilon_1 \Big)_+ \bigg\} \bigg) \nonumber \\
& \hspace{2cm} \bigg( 1- \exp \bigg\{ - \Big(\sum_{m \in A} \tilde{Z}_{K-i}^{(m,B)} g_2(x_m) - \epsilon_2 \Big)_+ \bigg\} \bigg) \eff_B(A)\bigg].
%&  \hspace{3cm} \prob(U_1^{(B)} >0 ) \prob \Big( Z_{K-i}^{(B)} > 0 \Big)^{|A|} \Big( \prob(Z_{K-i}^{(B)}=0)  \Big)^{\tilde{U}_1^{(B)} - |A|} \bigg]
\end{align}
Hence the right hand side of \eqref{eq:apll3} becomes,
\begin{align} \label{eq:apll6}
& \sum_{i=1}^K  \mu_B^{i-1} \int_{\real^B}  \exptn \bigg[  \sum_{A \in \pow([\tilde{U}_1^{(B)}]) \setminus \{\nullm\}} \bigg( 1- \exp \bigg\{ - \Big(\sum_{m \in A} \tilde{Z}_{K-i}^{(m,B)} g_1(x_m) - \epsilon_1 \Big)_+ \bigg\} \bigg) \nonumber \\
& \hspace{2cm} \bigg( 1- \exp \bigg\{ - \Big(\sum_{m \in A} \tilde{Z}_{K-i}^{(m,B)} g_2(x_m) - \epsilon_2 \Big)_+ \bigg\} \bigg) \eff_B(A) \bigg] \lambda^{(B)}(\dtv \bld{x}) \nonumber \\
%& \hspace{3cm} \prob(U_1^{(B)} >0 ) \prob \Big( Z_{K-i}^{(B)} > 0 \Big)^{|A|} \Big( \prob(Z_{K-i}^{(B)}=0)  \Big)^{\tilde{U}_1^{(B)} - |A|} \bigg] \lambda^{(B)}(\dtv \bld{x}) \nonumber \\
& = \sum_{i=0}^{K-1}  \mu_B^{K-i-1} \int_{\real^B}  \exptn \bigg[  \sum_{A \in \pow([\tilde{U}_1^{(B)}]) \setminus \{\nullm\}} \bigg( 1- \exp \bigg\{ - \Big(\sum_{m \in A} \tilde{Z}_{i}^{(m,B)} g_1(x_m) - \epsilon_1 \Big)_+ \bigg\} \bigg) \nonumber \\
& \hspace{2cm} \bigg( 1- \exp \bigg\{ - \Big(\sum_{m \in A} \tilde{Z}_{i}^{(m,B)} g_2(x_m) - \epsilon_2 \Big)_+ \bigg\} \bigg) \nonumber \\
& \hspace{3cm} \prob(U_1^{(B)} >0 ) \prob \Big( Z_i^{(B)} > 0 \Big)^{|A|} \Big( \prob(Z_i^{(B)}=0)  \Big)^{\tilde{U}_1^{(B)} - |A|} \bigg] \lambda^{(B)}(\dtv \bld{x})\nonumber\\
&=\int_{\scrmrz} \bigg( 1- \exp\bigg\{ - \Big(\int g_1 \dtv \nu - \epsilon_1 \Big)_+ \bigg\} \bigg) \bigg( 1- \exp\bigg\{ - \Big(\int g_2 \dtv \nu - \epsilon_2 \Big)_+ \bigg\} \bigg) \Upsilon(\dtv \nu),
%\end{align}
%Using a change of indices and rewriting the expression in terms of the probability distribution, the above can be represented as,
%\begin{align} \label{eq:apll7}
%& \sum_{i=0}^{K-1}  \mu_B^{K-i-1} \int_{\real^B} \sum_{t=1}^{B} \prob(\ub_1 =t) \sum_{A \in \pow([t]) \setminus \{ \nullm \}} \sum_{\stackrel{\bld{y} \in \bbn^t}{ y_m >0, m \in A, y_m=0, m \nin A }} \prod_{m \in [t]} \prob(Z_i^{(m,B)} = y_m) \nonumber \\
%& \bigg( 1- \exp\bigg\{ - \Big(\sum_{m \in A} y_m g_1(x_m) - \epsilon_1 \Big)_+ \bigg\} \bigg) \bigg( 1- \exp\bigg\{ - \Big(\sum_{m \in A} y_m g_2(x_m) - \epsilon_2 \Big)_+ \bigg\} \bigg) \lambda^{(B)}(\dtv \bld{x})\nonumber\\
%&=\int_{\scrmrz} \bigg( 1- \exp\bigg\{ - \Big(\int g_1 \dtv \nu - \epsilon_1 \Big)_+ \bigg\} \bigg) \bigg( 1- \exp\bigg\{ - \Big(\int g_2 \dtv \nu - \epsilon_2 \Big)_+ \bigg\} \bigg) \Upsilon(\dtv \nu)
\end{align}

 where $\Upsilon$ is a measure  on the space $\scrm_0$ defined as
\begin{align} \label{eq:apll10}
%\Upsilon(\cdot) & : = \sum_{i=0}^{K-1}  \mu_B^{K-i-1}  \sum_{t=1}^{B} \prob(\ub_1 =t) \sum_{A \in \pow(\bbn(t)) \setminus \{ \nullm \}} \sum_{\stackrel{\bld{y} \in \bbn^t}{ y_m >0, m \in A, y_m=0, m \nin A }} \nonumber \\
%& \hspace{2cm} \prod_{m \in \bbn(t)} \prob(Z_i^{(m,B)} = y_m)  \lambda^{(B)} (\bld{x} : \sum_{m=1}^t y_m \delta_{x_m} \in \cdot) \nonumber \\
\Upsilon(\cdot)& := \sum_{i=0}^{K-1}  \mu_B^{K-i-1} \exptn \bigg[  \sum_{A \in \pow([\tub_1]) \setminus \{ \nullm \}} \lambda^{(B)} (\bld{x} : \sum_{m \in A} Z_i^{(m,B)} \delta_{x_m} \in \cdot) \nonumber \\
& \hspace{2cm} \prob(\ub_1 >0) \Big( \prob(Z_i^{(B)} >0) \Big)^{|A|} \Big( \prob(Z_i^{(B)} =0)^{\tub_1 - |A| } \Big)\bigg]
\end{align}

It remains to verify that $\Upsilon(B_r) < \infty$ for every $r >0$, where $B_r$ is as in~\eqref{eq:defn_Br}.  Fix an $r >0$. We get
\aln{
\Upsilon(B_r)
&= \sum_{i=0}^{K-1} \mu_B^{K-i-1} \exptn \Big[ \sum_{A \in \pow([t]) \setminus \{\emptyset\}}  \lambda^{(B)} \{\mathbf{x} : \sum_{m \in A} \tilde{Z}_i^{(m,B)} \delta_{x_m} \in B_r\} \prob(U_1^{(B)} >0) \nonumber \\
& \hspace{2cm} (\prob(Z_i^{(B)} >0))^{|A|} (\prob(Z_i^{(B)} =0))^{t-|A|}\Big] \nonumber \\
& \le \sum_{i=0}^{K-1} \mu_B^{K-i-1} \sum_{t=1}^{B} \prob(U_1^{(B)} =t) \sum_{A \in \pow([t]) \setminus \{\emptyset\}}  \lambda^{(B)} \{\mathbf{x} : \sum_{m \in A}  \delta_{x_m} \in B_r\} \nonumber \\
& \hspace{2cm} (\prob(Z_i^{(B)} >0))^{|A|} (\prob(Z_i^{(B)} =0))^{t-|A|}. \label{eq:upsilon_radon_1}
}
Fix $i,t,A$. Then it is easy to see that $\sum_{m \in A} \delta_{x_m} \in B_r$ if there exists some $\eta(i,t,A) >0$ such that, for some $m \in A$, $x_m > \eta(i,t,A)$.  Hence using $\eta := \min_{0 \le i \le K-1,  1 \le t \le B, A \in \pow([t] \setminus \{ \nullm\})} \eta(i,t,A)$ and $\lambda^{(1)}\{x \in \real : |x| > \eta \} = \eta^{-\alpha}$, we have $\Upsilon(B_r)\le B \eta^{-\alpha} \sum_{i=0}^{K-1} \sum_{t=1}^B \prob(U_1^{(B)} =t) (1- \prob(Z_i=0)^t) < \infty$. This completes the proof of Lemma~\ref{lemma:treerv}. \qedhere

\subsection{Establishing the Weak Convergence: Proof of Lemma \ref{lemma:Limit_Laplace}}

%\begin{proof}[Proof of Lemma~\ref{lemma:Limit_Laplace}]
In this proof we shall give explicit expressions for $\nskb$ and $N_*^{(K)}$, and establish the weak convergence results assuming that these point processes (and also $N_*$) are Radon. The almost everywhere Radonness of these point processes will be established in Lemma~\ref{lemma:limit:radon}.

To show part (a), we shall compute the limiting Laplace functional of $\dnkb$ under $\pstar$ as $n \to \infty$, i.e., the limit of
\aln{
\estar \Big( \exp \{- \dnkb(f)\} \Big)
}
for a continuous and bounded function $f$ that vanishes in a neighbourhood of $0$. It is easy to see that $\dnkb = \mbfs_{c_n^\inv} \sum_{j=1}^{|D_{n-K}|} N_n^{' (K,B,j)}$,
where $N_n^{' (K,B,j)}$ denotes the point process associated to the $j^{th}$ subtree $\tilde{\bbt}^{(B)}_j$ without scaling, for $j=2, \ldots, |D_{n-K}|$ conditioned on $\calf_{n-K}$.
%It is also important to note that
%\beq \label{eq:apll13}
%\pstar(A) = \frac{\prob(A \cap \cals)}{\prob(\cals)} = \frac{1}{\prob(\cals)} \int \bbo_{A} \bbo_\cals \dtv \prob
%\eeq
%and we also know that $\pstar$ is absolutely continuous with respect to $\prob$. Using these facts and Radon-Nikodym theorem we get
%\beq \label{eq:apll14}
%\frac{\dtv \pstar}{\dtv \prob} = \frac{1}{\prob(\cals)} \bbo_{ \cals}.
%\eeq
Now, using the fact that $\dtv \pstar = \frac{1}{\prob(\cals)} \bbo_{ \cals}\dtv \prob$, where $\cals$ denotes the event that the Galton-Watson tree survives, it is easy to see that
\aln{
& \estar \bigg[\exp \bigg\{ - \dnkb(f) \bigg\} \bigg]
%& = \estar \bigg[ \exp \bigg\{ -  \mbfs_{c_n^\inv} \sum_{j=1}^{|D_{n-K}|} N_n^{' (K,B,j)} (f) \bigg\} \bigg] \nonumber \\
= \frac{1}{\prob(\cals)} \exptn \bigg[ \bbo_\cals \exp \bigg\{  -  \mbfs_{c_n^\inv} \sum_{j=1}^{|D_{n-K}|} N_n^{' (K,B,j)} (f) \bigg\} \bigg].  \label{eq:apll15}
}
Following arguments in \citet{bhattacharya:hazra:roy:2014},  it follows that to show the convergence of the right-hand side of \eqref{eq:apll15} it is enough to show the convergence of
\begin{equation}\label{eq:apll16}
\frac{1}{\prob(\cals)} \exptn \bigg[ \bbo_{(|D_{n-K}| >0 )}  \exptn \bigg( \exp \bigg\{  -  \mbfs_{c_n^\inv}\sum_{j=1}^{|D_{n-K}|} N_n^{' (K,B,j)} (f) \bigg\} | \calf_{n-K} \bigg) \bigg].
\end{equation}

%we would like to introduce another event $S_{n-K}$ which is empty if $|D_{n-K}|=0$ and otherwise denotes that there is at least one infinite tree rooted at $(n-K)^{th}$ generation. By introducing the set , it can easily be observed that $\bbo_\cals = \bbo_{(|D_{n-K}| >0 )} \bbo_{S_{n-K}}$ and the right hand side of \eqref{eq:apll15} becomes
%\aln{
%& \frac{1}{\prob(\cals)} \exptn \bigg[ \bbo_{(|D_{n-K}| >0 )} \bbo_{S_{n-K}} \exp \bigg\{  -  \mbfs_{c_n^\inv}\sum_{j=1}^{|D_{n-K}|} N_n^{' (K,B,j)} (f) \bigg\} \bigg] \nonumber \\
%&= \frac{1}{\prob(\cals)} \exptn \bigg[ \bbo_{(|D_{n-K}| >0 )}  \exp \bigg\{  -  \mbfs_{c_n^\inv}\sum_{j=1}^{|D_{n-K}|} N_n^{' (K,B,j)} (f) \bigg\} \bigg]  \nonumber \\
%&-  \frac{1}{\prob(\cals)} \exptn \bigg[ \bbo_{(|D_{n-K}| >0 )}  \bbo_{S_{n-K}^c} \exp \bigg\{  -  \mbfs_{c_n^\inv}\sum_{j=1}^{|D_{n-K}|} N_n^{' (K,B,j)} (f) \bigg\} \bigg]
%}
%Using the same argument as in \citet{bhattacharya:hazra:roy:2014}, it is clear that the second term vanishes as $n \to \infty$. So, it is enough to find the limit of the first term as $n \to \infty$ which is
%\aln{
%& \frac{1}{\prob(\cals)} \exptn \bigg[ \exptn \bigg( \bbo_{(|D_{n-K}| >0 )}  \exp \bigg\{  -  \mbfs_{c_n^\inv}\sum_{j=1}^{|D_{n-K}|} N_n^{' (K,B,j)} (f) \bigg\} \bigg| \calf_{n-K} \bigg) \bigg] \nonumber \\
%&= \frac{1}{\prob(\cals)} \exptn \bigg[ \bbo_{(|D_{n-K}| >0 )}  \exptn \bigg( \exp \bigg\{  -  \mbfs_{c_n^\inv}\sum_{j=1}^{|D_{n-K}|} N_n^{' (K,B,j)} (f) \bigg\} | \calf_{n-K} \bigg) \bigg] \label{eq:apll16}
%}
 It is important to note that, $\{ N_n^{' (K,B,j)} : j=1, \ldots, |D_{n-K}|\}$ conditioned on $\calf_{n-K}$ are an i.i.d. collection of point processes under the law $\prob$. Hence the conditional expectation in \eqref{eq:apll16} becomes
\begin{equation}\label{eq:apll17}
\prod_{j=1}^{|D_{n-K}|}\exptn \bigg( \exp \bigg\{  -  \mbfs_{c_n^\inv} N_n^{' (K,B,j)} (f) \bigg\} | \calf_{n-K} \bigg)=  \bigg[ \exptn \bigg( \exp \bigg\{  -  \mbfs_{c_n^\inv}\dnkbo (f) \bigg\} \bigg) \bigg]^{\mu^n \frac{Z_{n-K}}{\mu^n}}.
\end{equation}
%\aln{
%&\prod_{j=1}^{|D_{n-K}|}\exptn \bigg( \exp \bigg\{  -  \mbfs_{c_n^\inv} N_n^{' (K,B,j)} (f) \bigg\} | \calf_{n-K} \bigg) \nonumber \\
%&= \bigg[ \exptn \bigg( \exp \bigg\{  -  \mbfs_{c_n^\inv}\dnkbo (f) \bigg\} | \calf_{n-K} \bigg) \bigg]^{|D_{n-K}|} \nonumber \\
%&= \bigg[ \exptn \bigg( \exp \bigg\{  -  \mbfs_{c_n^\inv}\dnkbo (f) \bigg\} \bigg) \bigg]^{Z_{n-K}}  \nonumber \\
%& = \bigg[ \exptn \bigg( \exp \bigg\{  -  \mbfs_{c_n^\inv}\dnkbo (f) \bigg\} \bigg) \bigg]^{\mu^n \frac{Z_{n-K}}{\mu^n}}
%}
Combining the result in \eqref{eq:regvar_subtree} and the technique used in proof of Theorem \ref{thm:mainthm1}, it is easy to see that
\aln{
 \bigg[ \exptn \bigg( \exp \bigg\{  -  \mbfs_{c_n^\inv}\dnkbo (f) \bigg\} \bigg) \bigg]^{\mu^n} \to \exp \bigg\{ - \int_{\scrm_0} \Big( 1- \exp\{ \nu(f)\} \Big) \Upsilon(d\nu) \bigg\} \label{eq:apll18}
}
as $n \to \infty$, and  we know from \eqref{eq:martngle_conv} that $\mu^{-n}Z_{n-K}\to \mu^{-K} W$ almost surely as $n\to\infty$.
%\aln{
%\frac{Z_{n-K}}{\mu^n} = \frac{1}{\mu^k} \frac{Z_{n-K}}{\mu^{n-K}} \to \frac{1}{\mu^k} W \label{eq:apll19}
%}
%almost surely as $n \to \infty$.
Finally, using \eqref{eq:apll18}, we get that the almost sure limit of right-hand side of \eqref{eq:apll17} is
\aln{
 \exp \bigg\{ - \frac{1}{\mu^K} W\int_{\scrm_0} \Big( 1- \exp\{ \nu(f)\} \Big) \Upsilon(d\nu) \bigg\} \label{eq:apll20}
}
as $n \to \infty$.
Hence, using dominated convergence theorem and the fact that $\bbo_{|D_{n-K} >0|}$ converges almost surely to $\bbo_\cals$, it follows that
\aln{\estar \bigg[\exp \bigg\{ - \dnkb(f) \bigg\} \bigg] = \estar \bigg[\exp \bigg\{ - \frac{1}{\mu^K} W\int_{\scrm_0} \Big( 1- \exp\{ \nu(f)\} \Big) \Upsilon(d\nu) \bigg\} \bigg]. \label{eq:apll22}}

% It is easy to see that the limit and the right hand side of \eqref{eq:apll17} is bounded by $1$ and $\bbo_{|D_{n-K} >0|}$ converges almost surely to $\bbo_\cals$ and hence using dominated convergence theorem we immediately get that the right hand side \eqref{eq:apll16} converges to
%\aln{
%&\frac{1}{\prob(\cals)}\exptn \bigg[\bbo_{\cals}\exp \bigg\{ - \frac{1}{\mu^K} W\int_{\scrm_0} \Big( 1- \exp\{ \nu(f)\} \Big) \Upsilon(d\nu) \bigg\} \bigg] \nonumber \\
%& = \estar \bigg[\exp \bigg\{ - \frac{1}{\mu^K} W\int_{\scrm_0} \Big( 1- \exp\{ \nu(f)\} \Big) \Upsilon(d\nu) \bigg\} \bigg] \label{eq:apll21}
%}
%Finally using the right hand side of \eqref{eq:apll10}, we can write down the limiting Laplace functional as
%\aln{
%&\estar \Bigg[ \exp \Bigg\{ -W \frac{1}{\mu^K}\sum_{i=0}^{K-1} \mu_B^{K-i-1} \int_{\real^B \setminus \bld{0}_B} \exptn \bigg[ \sum_{A \in \pow(\bbn(\tub_1)) \setminus \{ \nullm \}} \Big( 1 - \exp \{- \sum_{m \in A} \tilde{Z}_i^{(m,B)} f(x_m)\} \Big) \nonumber \\
%& \hspace{2cm} \prob(U_1^{(B)} >0) \prob(Z_i^{(B)} >0)^{|A|} \prob(Z_i^{(B)} =0)^{t-|A|} \bigg]  \lambda^{(B)}( \dtv \bld{x})\Bigg\} \Bigg].  \label{eq:apll22}
%}
Next, we shall produce a point process that has the Laplace functional as in \eqref{eq:apll22}. This description is similar to the description of $N_*$. Let
\alns{
\poi_B := \sum_{l=1}^\infty \delta_{\xi_{l,1}, \ldots, \xi_{l,B}} := \sum_{l=1}^\infty \delta_{\pmb{\xi}_l}
}
be a Poisson random measure on $\real^B \setminus \{\pmb{0}_B\}$ with intensity measure $\lambda^{(B)}$ and independent of $W$ (see \eqref{eq:martngle_conv}). Let $V_B$ be an $\{1, \ldots, B\}$-valued random variable with probability mass function
\aln{
\prob ( V_B = t) = \frac{1}{s_B} \prob(Z_1^{(B)} = t) \sum_{i=0}^K \frac{\mu_B^{K-i-1}}{\mu^K} \Big( 1 - (\prob(Z_i^{(B)}=0))^t \Big) \label{eq:pmf_VB}
}
where $s_B$ is the normalizing constant.
% with
%\aln{
%s_B := \sum_{t=1}^B \sum_{i=0}^{K-1}  \frac{\mu_B^{K-i-1}}{\mu^K} \Big( 1 - (\prob(Z_i^{(B)}=0))^t \Big). \label{eq:defn_SB}
%}
Suppose that $\mathbf{T}^{(B)} = (T_1^{(B)}, \ldots, T_B^{(B)})$ is an $\bbn_0^B$-valued random variable with  mass function conditioned on the random variable $V_B$,
\aln{
\prob(\mathbf{T}^{(B)} = \mathbf{y} | V_B= t) = \begin{cases} &0  \mbox{ if } \mathbf{y} = \mathbf{0}_B  \mbox{ or for some } t < k \le B, y_k >0, \\ & \frac{1}{s_t} \sum_{i=0}^{K-1} \frac{\mu_B^{K-i-1}}{\mu^K} \prod_{m=1}^t \prob( Z_i^{(B)} = y_m) \mbox{ otherwise}  \end{cases}  \label{eq:cpmf_TB}
}
where $\mathbf{y} = (y_1, \ldots,y_B) \in \bbn_0^B$, $t \in \{1, \ldots, B\}$ and $s_t$ is again a normalising constant.
%\aln{
%s_t = \sum_{(y-1, \ldots, y_t) \neq \mathbf{0}_t} \sum_{i=0}^{K-1} \frac{\mu_B^{K-i-1}}{\mu^K} \prod_{m=1}^t \prob( Z_i^{(B)} = y_m) \label{eq:defn_st}
%}
%with $\mathbf{0}_t \in \bbn_0^t$.
Finally, consider the collection $\{(V_l^{(B)}, \mathbf{T}_l^{(B)} ) : l \in \bbn\}$ of  independent copies of $(V_B, \mathbf{T}^{(B)})$ and also independent of $W$ and $\poi_B$. Now consider the following point process
\aln{
N_*^{(K,B)} = \sum_{l=1}^{\infty} \sum_{k=1}^{V_l^{(B)}} T_{l,k} \delta_{(s_BW)^{1/\alpha} \xi_{l,k}} . \label{eq:defn_nskb}
}
 We want to compute the Laplace functional of this point process and verify that it the same as in the expression of \eqref{eq:apll22}.

We shall compute the Laplace functional of $N_*^{(K,B)}$ by computing the Laplace functional of an auxiliary marked Cox process. Define the auxiliary point process as 
\aln{
\mathscr{P}_*^{(K,B)} = \sum_{l=1}^{\infty} \delta_{(V_l^{(B)}, T_{l,1}^{(B)} , \ldots, T_{l,B}^{(B)}, (s_B W)^{1/\alpha} \xi_{l,1}, \ldots,(s_B W)^{1/\alpha} \xi_{l,B})}. \label{eq:defn_pskb}
}
We want to consider a function $f$ on the metric space $(\{1, \ldots, B\} \times \bbn_0^B \times \real^{B} , d_B)$ that is bounded, continuous and vanishes on a neighbourhood of the set $\{1, \ldots, B\} \times \bbn_0^B \times \{\mathbf{0}_B\}$. Then the Laplace functional  will be
\aln{
& \estar \Bigg[ \exp \bigg\{ - \mathscr{P}_*^{(K,B)}(f) \bigg\} \Bigg]
%&= \estar \Bigg[ \exp \bigg\{ - \sum_{l=1}^\infty f(V_l^{B}, \mathbf{T}_l^{(B)}, (s_B W)^{1/\alpha} \pmb{\xi}_l) \bigg\} \Bigg] \nonumber \\
= \estar \Bigg[  \exptn \Bigg(\exp \bigg\{ - \sum_{l=1}^\infty f(V_l^{B}, \mathbf{T}_l^{(B)}, (s_B W)^{1/\alpha} \pmb{\xi}_l) \bigg\} \bigg| W \Bigg)  \Bigg].  \label{eq:lap_nskb1}
}
Use the fact that, conditioned on $W$, $\mathscr{P}_*^{(K,B)}$ is a marked Poisson point process with i.i.d. marks $\{(V_l^{(B)}, \mathbf{T}_l^{(B)}) : l \in \bbn\}$ which are also independent of the Poisson points $\{\pmb{\xi}_l : l \in \bbn\}$. Following Proposition~3.8 of \citet{resnick:1987}, we get that the right-hand side of the above equation equals
\aln{
\estar \Bigg[ \exp \bigg\{ - \int_{\real^B} \exptn \bigg[ 1- \exp \Big\{ - f(V_1^{(B)}, \mathbf{T}_1^{(B)}, (s_B W)^{1/\alpha} \mathbf{x}) \Big\} \bigg] \lambda^{(B)} (\dtv \mathbf{x}) \bigg\} \Bigg]. \label{eq:lap_nskb2}
}
Recall that $\lambda^{(B)}$ satisfies $\lambda^{(B)} (a \cdot) = a^{-\alpha} \lambda^{(B)}(\cdot)$
for every $a>0$. Hence \eqref{eq:lap_nskb2} equals
\aln{
\estar \Bigg[ \exp \bigg\{ - W s_B\int_{\real^B} \exptn \bigg[ 1- \exp \Big\{ - f(V_1^{(B)}, \mathbf{T}_1^{(B)},  \mathbf{x}) \Big\} \bigg] \lambda^{(B)} (\dtv \mathbf{x}) \bigg\} \Bigg]. \label{eq:lap_nskb3}
}
Next we compute the Laplace functional of $N_*^{(K,B)}$. Let $f\in \ckr$ and  choose a function $f' : \{1, \ldots, B\} \times \bbn^B \times \real^B \to \real^+$ such that
\aln{
f'(t, y_1, \ldots, y_B, x_1, \ldots, x_B) = \sum_{m=1}^B y_m f(x_m) \label{eq:lap_nskb4}
}
for every $t \in \bbn$, $y_i \in \bbn_0$ and $x_i \in \real$ for every $i=1,2, \ldots, B$. Then
\aln{
& \estar \Bigg[ \mathscr{P}_*^{(K,B)}(f') \Bigg]  = \estar \Bigg[ \exp \bigg\{ - \sum_{l=1}^\infty \sum_{m=1}^{V_l^{(B)}} T_{l,m} f(\xi_{l,m}) \bigg\} \Bigg]  = \estar \Bigg[ \exp \bigg\{ - N_*^{(K,B)}(f)\bigg\} \Bigg] \nonumber \\
& = \estar \Bigg[ \exp \bigg\{ - s_B W  \int_{\real^B} \exptn \Big( 1- \exp\{ - \sum_{m=1}^{V_1^{(B)}} T_{1,m} f(x_m)\} \Big) \lambda^{(B)} (\dtv \mathbf{x}) \bigg\} \Bigg]. \label{eq:lap_nskb5}
}
We compute the expectation in the exponent, discounting the $0$'s, as follows:
\aln{
& \exptn \Bigg[ 1 - \exp \bigg\{ - \sum_{m=1}^{V_1^{(B)}} T_{1,m} f(x_m) \bigg\} \Bigg]
= \sum_{t=1}^B \prob(V_1^{(B)} =t) \exptn\Bigg[ 1- \exp \bigg\{ - \sum_{m=1}^t T_{1,m} f(x_m) \bigg\}  \bigg| V_1^{(B)} =t \Bigg] \nonumber \\
%&= \sum_{t=1}^B \prob(V_B =t)  \frac{1}{s_t} \sum_{i=0}^{K-1} \frac{\mu_B^{K-i-1}}{\mu^K} \sum_{\mathbf{y} \in \bbn_0^t \setminus \{ \mathbf{0}_t\}} \Bigg( 1- \exp \bigg\{ -\sum_{m=1}^t y_m g(x_m)  \bigg\} \Bigg) \prod_{m=1}^t \prob(Z_i^{(B)} =y_m) \nonumber \\
&=\sum_{i=0}^{K-1} \frac{\mu_B^{K-i-1}}{\mu^K} \frac{1}{s_B} \sum_{t=1}^B \prob(Z_1^{(B)} =t) s_t \frac{1}{s_t} \sum_{A \in \pow ([t]) \setminus \{\emptyset\}} \Bigg(1- \exp \bigg\{ - \sum_{m \in A} y_m f(x_m) \bigg\} \Bigg) \nonumber \\
&  \hspace{2cm} \prod_{m \in A} \prob(\tilde{Z}_i^{(B)} = y_m) \Big( \prob(Z_i^{(B)} =0) \Big)^{|A|} \Big( \prob(Z_i^{(B)} =0) \Big)^{t- |A|} \nonumber \\
&= \frac{1}{\mu^K s_B} \sum_{i=0}^{K-1} \mu_B^{K-i-1} \exptn \Bigg[ \sum_{A \in \pow([\tilde{U}^{(B)}_1]) \setminus \{\emptyset\}} \bigg(  1- \exp \bigg\{ - \sum_{m \in A} \tilde{Z}_i^{(m,B)} f(x_m) \bigg\} \bigg) \nonumber \\
& \hspace{2cm} \Big(\prob(Z_i^{(B)} =0) \Big)^{|A|} \Big(\prob(Z_i^{(B)}=0) \Big)^{\tilde{U}^{(B)}_1- |A|}\Bigg] \prob(U_1^{(B)} >0). \label{eq:lap_nskb6}
}
Hence, combining the expressions in \eqref{eq:lap_nskb5} and \eqref{eq:lap_nskb6} and the definition of $\Upsilon$, it is easy to verify that the Laplace functional of $N_*^{(K,B)}$ is same as in \eqref{eq:apll22}. This completes the proof of part(a).

To show (b) we  let $B \to \infty$ and use Theorem 4.1 of \citet{lindskog:resnick:roy:2014} to get that the right-hand side of \eqref{eq:apll22} converges to
\aln{
& \estar \Bigg[ \exp \bigg\{ - W \sum_{i=0}^{K-1} \frac{1}{\mu^{i+1}} \int_{\real^\bbn \setminus \{ \mathbf{0}\}} \exptn \bigg[ \sum_{A \in \pow(\bbn(U_1)) \setminus \{ \emptyset \}} \bigg( 1 - \exp \Big\{ - \sum_{m \in A} \tilde{Z}_i^{(m)} f(x_m) \Big\} \bigg) \nonumber \\
& \hspace{2cm}\prob(U_1>0) \Big( \prob(Z_i > 0) \Big)^{|A|} \Big( \prob(Z_i=0) \Big)^{U_1 - |A|} \bigg] \lambda(\dtv \mathbf{x}) \bigg\} \Bigg] \label{eq:lap_B_infty}
}
Here, $U_1 \eqd Z_1$ and is independent of $W$, $\{Z_i^{(m)} : m \in  \bbn \}$ is a collection of independent copies of $Z_i$, which is also independent of $W$ and $U_1$ for every fixed $i \in \bbn_0$, $\{\{Z_i^{(m)}\}_{m \in \bbn} : i \in \bbn_0\}$ are independent sequences of random variables and $\lambda$ is introduced in \eqref{eq:jtregvar}. We shall construct another point process with same Laplace functional as in \eqref{eq:lap_B_infty}. Recall $\poi$ from \eqref{eq:defn_poi}, which is independent of $W$. We can define random variables $V^{(K)}$ and  $\mathbf{T}^{(K)} =(T_i^{(K)} : i \in \bbn) $  an $\bbn_0^{\bbn}$-valued random variable by replacing $\mu_B$ and $Z_i^{(B)}$ with $\mu$ and $Z_i$, respectively, in equations~\eqref{eq:pmf_VB} and~\eqref{eq:cpmf_TB}.
%with  with the following mass function conditioned on the random variable $V^{(K)}$,  with the following probability mass function
%\alns{
%\prob(V^{(K)} =t) = \frac{1}{s_K} \prob(Z_1=t) \sum_{i=0}^{K-1} \frac{1}{\mu^{i+1}} \Big(1 - (\prob(Z_i=0))^t \Big)
%}
%where $s_K$ is the normalizing constant with
%\alns{
%s_K = \sum_{t \in \bbn} \sum_{i=0}^{K-1} \frac{1}{\mu^{i+1}} \Big(1 - (\prob(Z_i=0))^t \Big).
%}
%Suppose $\mathbf{T}^{(K)} =(T_i^{(K)} : i \in \bbn) $ is an $\bbn_0^{\bbn}$-valued random element  with the following mass function conditioned on the random variable $V^{(K)}$,
%\alns{
%\prob(\mathbf{T}^{(K)} = \mathbf{y} | V^{(K)} =t) = \begin{cases} & 0  \mbox{  if } \mathbf{y} = \mathbf{0} \mbox{ or for some } k >t, y_k >0 \\
%& \frac{1}{s_{t,K}} \sum_{i=0}^{K-1} \mu^{-i-1} \prod_{m=1}^t \prob(Z_i= y_m) \mbox{ otherwise} \end{cases}
%}
%where $\mathbf{y} =(y_i : i \in \bbn) \in \bbn_0^\bbn $, $t \in \bbn$ and
%\alns{
%s_{t,K} = \sum_{\mathbf{y} : (y_1, y_2, \ldots, y_t) \neq \mathbf{0}_t} \sum_{i=0}^{K-1} \frac{1}{\mu^{i+1}} \prod_{m=1}^t \prob(Z_i = y_m)
%}
%with $\mathbf{0}_t \in \bbn_0^t$. Finally we
Consider the collection of random variables $\{(V_l^{(K)}, \mathbf{T}_l^{(K)}) : l \in \bbn\}$, which are independent copies of $(V^{(K)}, \mathbf{T}^{(K)})$ and also independent of $W$ and $\poi$. Now define the following point process
\alns{
N_*^{(K)} = \sum_{l=1}^{\infty} \sum_{m=1}^{V^{(K)}_l} T_{l,m}^{(K)} \delta_{(s_K W)^{1/\alpha} \xi_{l,m}}.
}
Again we can easily verify that the point process $N_*^{(K)}$ has the Laplace functional as in \eqref{eq:lap_B_infty} by computing the Laplace functional following the same steps as we have done for $N_*^{(K,B)}$. This completes the proof of part (b).

Finally, to show (c) we argue as follows. It is easy to see that as $K\to \infty$,  the right-hand side of \eqref{eq:lap_B_infty} becomes
\aln{
& \estar \Bigg[ \exp \bigg\{ - W \sum_{i=0}^{\infty} \frac{1}{\mu^{i+1}} \int_{\real^\bbn \setminus \{ \mathbf{0}\}} \exptn \bigg[ \sum_{A \in \pow([U_1]) \setminus \{ \emptyset \}} \bigg( 1 - \exp \Big\{ - \sum_{m \in A} \tilde{Z}_i^{(m)} f(x_m) \Big\} \bigg) \nonumber \\
& \hspace{2cm}\prob(U_1>0) \Big( \prob(Z_i > 0) \Big)^{|A|} \Big( \prob(Z_i=0) \Big)^{U_1 - |A|} \bigg] \lambda(\dtv \mathbf{x}) \bigg\} \Bigg] \label{eq:lap_final_dep_brw}
}
and it can similarly be verified that this is the Laplace functional of $N_*$. Using the homogeneity property stated in \eqref{eq:homregvar}, we get that, for every $b>0$,
\alns{
& \Psi_{N_*}(g \|b) =\estar\bigg[ \exp\Big\{ - N_*(\mbfs_b g) \Big\}\bigg]  \\
&= \estar \Bigg[ \exp \bigg\{ - b^{-\alpha} W \sum_{i=0}^{\infty} \frac{1}{\mu^{i+1}} \int_{\real^\bbn \setminus \{ \mathbf{0}\}} \exptn \bigg[ \sum_{A \in \pow([U_1]) \setminus \{ \emptyset \}} \bigg( 1 - \exp \Big\{ - \sum_{m \in A} \tilde{Z}_i^{(m)} f(x_m) \Big\} \bigg) \nonumber \\
& \hspace{2cm}\prob(U_1>0) \Big( \prob(Z_i > 0) \Big)^{|A|} \Big( \prob(Z_i=0) \Big)^{U_1 - |A|} \bigg] \lambda(\dtv \mathbf{x}) \bigg\} \Bigg].
}
Hence, using Proposition \ref{sub:zet:prop2}, we can say that $N_*$ admits an SScDPPP representation. This completes the proof,  except that we have to verify that $N_*$, $\nskb$ and $N_*^{(K)}$ are Radon. This is the content of the next lemma. %\qedhere
%\end{proof}

\begin{lemma}\label{lemma:limit:radon}
$ N_*$, $\nskb$ and $N_*^{(K)}$ are random elements of $\scrm(\bar{\real}_0)$.
\end{lemma}
\begin{proof}
We shall give the proof for $N_*$. The other two cases can be done similarly. Let $A \subset \real$ be bounded away from $0$. It is enough to show that $N_*(A) < \infty$ almost surely. It is clear that if we can show that, conditioned on the random variable $W$, there are only finitely many Poisson $j_{l,m}$ points in the set $A$, then we are done because then $N_*(A)$ is a finite sum of the corresponding random variables $T_{l,m}$.  Our first step will be to show that $\exptn(M(A)) < \infty$, where
\beq \label{eq:polf53}
M := \sum_{l=1}^\infty \sum_{m=1}^{V_l} \delta_{j_{l,m}}.
\eeq
Let
$$A_m = \real_0 \times \real_0 \times \ldots\times \underbrace{A}_{m-th ~ position} \times \real_0 \times \ldots \subset \real^\infty \setminus \{0_\infty\},$$
which is bounded away from $0_\infty$. It is clear that $M_m (A) :=\sum_{l=1}^\infty \delta_{j_{l,m}} (A) = \sum_{l=1}^\infty \delta_{\bld{j}_l} (A_m)$.
%\beq \label{eq:polf54}
%M_m (A) :=\sum_{l=1}^\infty \delta_{j_{l,m}} (A) = \sum_{l=1}^\infty \delta_{\bld{j}_l} (A_m)
%\eeq
%and as $A$ is bounded away from $0$, hence $A_m$ is also bounded away from $0_\infty$.
Hence, $M_m(A)$ is a Poisson random variable for every $m \ge 1$, with mean $\lambda(A_m) = \lambda^{(1)}(A)$, as we have assumed that the marginals of the measure $\lambda(\cdot)$ are same. Also observe that  $\sum_{l=1}^{\infty} \sum_{m=1}^{V_l} \delta_{j_{l,m}} (A) =\sum_{m=1}^\infty \sum_{l : V_l \ge m} \delta_{j_{l,m}}(A)$.
%\beq \label{eq:polf55}
%\sum_{l=1}^{\infty} \sum_{m=1}^{V_l} \delta_{j_{l,m}} (A) =\sum_{m=1}^\infty \sum_{l : V_l \ge m} \delta_{j_{l,m}}(A)
%\eeq
 From the fact that $M_m(A) \sim \mbox{ Poisson}(\lambda^{(1)}(A))$, it is clear that
\beq \label{eq:polf56}
\sum_{l : V_l \ge m} \delta_{j_{l,m}} (A) \sim \mbox{ Poisson}( \prob(V_1 \ge m) \lambda^{(1)}(A))
\eeq
by an independent thinning of a Poisson point process. Hence we get that
\begin{align} \label{eq:polf57}
 \exptn \Big( \sum_{l=1}^\infty \sum_{m=1} ^{V_l} \delta_{j_{l,m}}(A) \Big)  & = \exptn \Big( \sum_{m=1}^\infty \sum_{l: V_l \ge m} \delta_{j_{l,m}} (A) \Big)
  = \sum_{m=1}^\infty \exptn \Big( \sum_{l: V_l \ge m} \delta_{j_{l,m}} (A) \Big)
 = \sum_{m=1}^\infty \prob(V_1 \ge m) \lambda^{(1)}(A) \nonumber \\
&= \lambda^{(1)}(A) \sum_{m=1}^\infty \prob(\tilde{Z}_1 \ge m)
%& = \lambda^{(1)}(A) \sum_{m=1}^{\infty} \frac{1}{\prob(Z_1 >0)} \prob(Z_1 \ge m) \nonumber \\
 = \frac{\mu}{\prob(Z_1 >0)} \lambda^{(1)}(A) < \infty.
\end{align}
 We can ignore the constants and see that
\begin{align} \label{eq:polf58}
\exptn \bigg( \sum_{l=1}^\infty \sum_{m=1}^{V_l} \delta_{W^{1/\alpha} j_{l,m}}(A)\bigg) & = \exptn \bigg( \exptn \bigg(\sum_{l=1}^\infty \sum_{m=1}^{V_l} \delta_{W^{1/\alpha} j_{l,m}} (A) \bigg| W\bigg) \bigg) \nonumber \\
& = \exptn \Big( \frac{\mu}{\prob(Z_1 >0)} \lambda^{(1)}( \mbfs_{W^{-1/\alpha}}A) \Big)
= \exptn \Big( W \frac{\mu}{\prob(Z_1 >0)} \lambda^{(1)}(A) \Big) \nonumber \\
&  = \frac{\mu}{\prob(Z_1 >0)} \lambda^{(1)}(A) < \infty.
\end{align}
 Hence we are done with the fact that $N_*$ is a Radon measure, and so an element from $\scrm(\bar{\real}_0)$. \qedhere

\end{proof}

\section{Consequences of Theorem~\ref{thm:mainthm2}} \label{sec:conseq}

\subsection{Proof of Corollary \ref{thm:maxima}}

Recall from Section~\ref{sec:prelim_main} that $M_n$ denotes the position of the rightmost particle of the $n^{th}$ generation. It is easy to see that, for every $x>0$,
\aln{
\lim_{n \to \infty} \pstar(c_n^{\inv} M_n < x) = \lim_{n \to \infty} \pstar(N_n (x,\infty) =0) = \pstar(N_*(x,\infty) =0).
\label{eq:maxima_point_process}}
So it is enough to compute the probability $\pstar(N_*(x,\infty) =0)$. Since $N_*$ is a Cox cluster process, we first condition on $W$, to get underlying Poisson random measure.  The right-hand side of \eqref{eq:maxima_point_process} becomes
\aln{
\estar \Bigg[ \prob \Bigg( \sum_{l=1}^\infty \sum_{k=1}^{V_l} T_{lk} \delta_{(sW)^{1/\alpha} \xi_{lk}} (x,\infty] =0 \Bigg| W \Bigg) \Bigg]. \label{eq:maxima_conditioned_W}
}
We next introduce some notation that will be useful later. Let $G_0= [-\infty, 1]$ and $G_1=(1,\infty]$. Assume that $H^t_{i_1, \ldots, i_t} \subset \real^\bbn$ denotes the set
\alns{
G_{i_1} \times G_{i_2} \times \cdots \times G_{i_t} \times \real \times \cdots
}
such that $i_j \in \{0,1\}$ for $1 \le j \le t$. It is clear that, for a fixed $t$, $\{H^t_{i_1, \ldots,i_t}: i_j \in {0,1}, 1\le j \le t\}$ is a collection of disjoint sets. To each of the sets $H^t_{i_1, \ldots,i_t}$ we associate a set $R^t_{i_1, \ldots,i_t}:= \{1 \le j \le t : i_j=1\}$. We introduce the set $O^t_{i_1, \ldots,i_t} = \prod_{j=1}^t N_{i_j} \prod_{j > t}{\bbn_0}$, where $N_{i_k} = \bbn_0$ if $i_k=0$ and $\prod_{j \in R^t_{i_1, \ldots,i_t}} N_{i_j} = \bbn_0^{|R^t_{i_1, \ldots, i_t}|} \setminus \{0_{|R^t_{i_1, \ldots,  i_t}|}\}$. The conditional probability inside the expectation in  \eqref{eq:maxima_conditioned_W} becomes
\aln{
& \prob \Bigg( \sum_{l=1}^\infty \sum_{v=1}^\infty \sum_{k=1}^{V_l} T_{lk} \delta_{V_l}(\{v\}) \delta_{(sW)^{1/\alpha} \xi_{lk}} (x,\infty] =0 \Bigg| W  \Bigg) \nonumber \\
&= \prob \Bigg( \sum_{v=1}^\infty \sum_{l=1}^\infty  \sum_{k=1}^{v} T_{lk} \delta_{V_l}(\{v\}) \delta_{(sW)^{1/\alpha} \xi_{lk}} (x,\infty] =0 \Bigg| W  \Bigg) \nonumber \\
&= \prob \Bigg( \sum_{v=1}^\infty \sum_{\stackrel{i_1,  \ldots, i_v}{ i_1+i_2+\cdots + i_v >0}}\sum_{l=1}^\infty  \delta_{\Big(V_l, \mathbf{T}_l, (sW)^{1/\alpha} x^\inv \pmb{\xi}_l\Big)} \bigg( \{v\} \times O^v_{i_1, \ldots,i_v} \times H^v_{i_1, \ldots, i_v}\bigg)  =0 \Bigg| W  \Bigg) \nonumber \\
&= \prob\Bigg[ \bigcap_{v=1}^\infty \bigcap_{\stackrel{i_1,  \ldots, i_v}{ i_1+i_2+\cdots + i_v >0}} \Bigg( \sum_{l=1}^\infty \delta_{\Big(V_l, \mathbf{T}_l, (sW)^{1/\alpha} x^\inv \pmb{\xi}_l\Big)} \bigg( \{v\} \times O^v_{i_1, \ldots,i_v} \times H^v_{i_1, \ldots, i_v}\bigg)  =0   \Bigg) \Bigg| W \Bigg]. \label{eq:maxima_poisson}
}
It is important to note that $\{v\} \times O^v_{i_1, \ldots,i_v} \times H^v_{i_1, \ldots, i_v}$ is a collection of disjoint sets over $v$ and $(i_1, \ldots,i_v)$. Using the fact that
\aln{
\sum_{l=1}^\infty \delta_{\Big( V_l, \mathbf{T}_l, (sW)^{1/\alpha} x^{-1} \pmb{\xi}_l \Big)} \Bigg( \{v\} \times O^v_{i_1, \ldots,i_v} \times H^v_{i_1, \ldots, i_v} \Bigg) \Bigg| W \sim \mbox{Poisson }\Bigg( sW x^{-\alpha}\prob \Big( \mathbf{T} \in O^v_{i_1, \ldots,i_v}, V=v \Big) \lambda(H^v_{i_1, \ldots,i_v}) \Bigg)
}
the right-hand side of \eqref{eq:maxima_poisson} becomes
\aln{
\exp \Bigg\{ -sW x^{-\alpha}\sum_{v=1}^{\infty} \sum_{\stackrel{i_1,\ldots,i_v}{i_1+\cdots + i_v >0} } \prob \Big( \mathbf{T} \in O^v_{i_1, \ldots,i_v}, V=v \Big) \lambda(H^v_{i_1, \ldots,i_v})  \Bigg\}. \label{eq:maxima_exponent_conditional}
}
We shall find a closed form expression for the exponent and after that we shall show that the exponent of \eqref{eq:maxima_exponent_conditional} is finite. Using the exchangeability property of $T_1, \ldots, T_v$ conditioned on the event $V=v$, for $|R^t_{i_1, \ldots,i_v} |=k$ we get 
\alns{
\prob(\mathbf{T} \in O^v_{i_1, \ldots, i_v} |V=v) &= \prob( \mbox{ at least one of } T_1, \ldots,T_k \mbox{ is positive} |V=v) \nonumber \\
&= \frac{1}{s_v} \sum_{i=0}^\infty \frac{1}{\mu^{i+1}} \sum_{\stackrel{y_1, \ldots, y_v}{ Y_1 + \cdots +y_k >0}} \prod_{m=1}^v \prob(Z_i= y_m) \nonumber \\
&= \frac{1}{s_v} \sum_{i=0}^\infty \frac{1}{\mu^{i+1}} \bigg( 1- \Big(\prob(Z_i=0) \Big)^k \bigg) =\frac{s_k}{s_v}.
}
Hence the sum in the exponent of \eqref{eq:maxima_exponent_conditional} becomes
\aln{
&\frac{1}{s} \sum_{v=1}^\infty \prob(Z_1=v)s_t \frac{1}{s_v} \sum_{k=1}^v s_k \sum_{\stackrel{i_1,\ldots,i_v}{i_1+ \cdots +i_v >0, |R^t_{i_1, \ldots, i_v}|=k }} \lambda(H^v_{i_1, \ldots , i_v}) \label{eq:maxima_final_form} \\
&\le \frac{1}{s(\mu-1)} \sum_{v=1}^\infty \prob(Z_1=v) \sum_{j=1}^v \sum_{\stackrel{i_1, \ldots, i_v}{ i_j=1}} \lambda(H^v_{i_1, \ldots,i_v}) \nonumber \\
& \le \frac{1}{s(\mu-1)} \sum_{v=1}^\infty \prob(Z_1=v) v \lambda^{(1)}(G_1) < \infty, \nonumber
}
using the facts that $s_k \le (\mu-1)^\inv$, the projections of $\lambda(\cdot)$ are identical and $\bigcup_{i_1, \ldots,i_v : i_j=1} H^v_{i_1, \ldots,i_v} \subset \prod_{i=1}^{j-1} \real \times G_1 \times \prod_{i>j+1} \real$. Finally,  combining \eqref{eq:maxima_exponent_conditional} and \eqref{eq:maxima_final_form}, we get \eqref{eq:maxima_lemma} with \label{not:kappa_lambda}
\beq
\kappa_\lambda = \sum_{v=1}^\infty \prob(Z_1=v) \sum_{k=1}^v s_k \sum_{\stackrel{i_1,\ldots,i_v}{i_1+ \cdots +i_v >0, |R^v_{i_1, \ldots, i_v}|=k }} \lambda(H^v_{i_1, \ldots , i_v}). \label{eq:maxima_c_lambda} \qedhere
\eeq

\begin{remark}  Theorem \ref{thm:maxima} is an extension of the main result of \citet{durrett:1983} to a dependent setup. Using the fact that $\lambda(\cdot) = \lamiid(\cdot)$, it is easy to get the asymptotic distribution of the maxima in case of branching random walk with regularly varying independent step sizes.  It is easy to see that in this case
$$\kappa_\lambda = \sum_{i=0}^\infty \frac{1}{\mu^i} \prob(Z_i >0).$$
The later can also be obtained from Corollary \ref{propn:mainresult_bhr} below; see Theorem~2.5 of \citet{bhattacharya:hazra:roy:2014}.
\end{remark}

\subsection{I.I.D.\ Displacements}
 In \citet{bhattacharya:hazra:roy:2014} the model is considered with i.i.d.  displacement random variables. Using Theorem \ref{thm:mainthm2}, with $\lambda= \lamiid$ (see \eqref{eq:iidprocess_lambda}), we get Theorem 2.1 of \citet{bhattacharya:hazra:roy:2014}.

\begin{cor}[Theorem 2.1 and Theorem 2.3 in \citet{bhattacharya:hazra:roy:2014}]  \label{propn:mainresult_bhr}
Under the assumptions of Theorem \ref{thm:mainthm2} and $\lambda=\lamiid$, for every $g \in \ckr$,
\aln{
& \estar [\exp\{-N_*(g)\}] \nonumber \\
& = \estar \bigg[  \exp \bigg\{ -  W \sum_{i=0}^\infty \frac{1}{\mu^i} \exptn \bigg[  \int_{\real} \bigg( 1 - \exp \bigg\{ - \tilde{Z}_i g(x) \bigg\} \bigg) \prob(Z_i >0 )    \nu_\alpha(\dtv x) \bigg]  \bigg\} \bigg]. \label{eq:laplacefun_iidcase}
}
In particular, $N_* \sim \mbox{SScDPPP}(m_\alpha, T\delta_{\varepsilon}, (rW)^{1/\alpha} )$ where $\varepsilon$ is an $\pm 1$-valued random variable with $\prob(\varepsilon =1) = p$, $m_\alpha$ is a measure on $(0,\infty)$ that is the same as $\nu_\alpha$ with $q=0$, and $T$ is a positive integer-valued random variable with probability mass function
\alns{
\prob(T=y) = \frac{1}{r} \sum_{i=0}^\infty \frac{1}{\mu^i} \prob(Z_i = y)
}
with $r = \sum_{i=0}^\infty \mu^{-i} \prob(Z_i >0)$.
\end{cor}

\begin{proof}

We  start with the exponent in \eqref{eq:lap_final_dep_brw} and shall show that it is same as that of the i.i.d.\ case as described in~\eqref{eq:laplacefun_iidcase}. For every $i \ge 1$, using the expression for $\lamiid(\cdot)$ in exponent of \eqref{eq:lap_final_dep_brw}, we get
\begin{align*}
& \exptn \bigg[ \sum_{l=1}^\infty \int_{\rnz} \sum_{A \in \pow([\tilde{U}_1]) \setminus \{\emptyset\}} \bigg( 1 - \exp \bigg\{ -\sum_{m \in A} \tilde{Z}_i^{(m)} g(x_m) \bigg\} \bigg)  \nonumber \\
& \hspace{3cm} \prob \Big( Z_1 > 0 \Big) \Big(\prob(Z_i >0 ) \Big)^{|A|} \Big( \prob(Z_i =0) \Big)^{\tilde{U}_1 - |A|}  \gamma_l(\dtv \bld{x}) \bigg] \nonumber \\
& = \exptn \bigg[ \sum_{l=1}^\infty \int_{\real} \sum_{\stackrel{A \in \pow([\tilde{U}_1]) \setminus \{\emptyset\}}{l\in A}} \bigg( 1 - \exp \bigg\{ - \tilde{Z}_i^{(l)} g(x_l) \bigg\} \bigg)  \nonumber \\
& \hspace{3cm} \prob \Big( Z_1 > 0 \Big) \Big(\prob(Z_i >0 ) \Big)^{|A|} \Big( \prob(Z_i =0) \Big)^{\tilde{U}_1 - |A|}  \nu_\alpha(\dtv x_l) \bigg].
\end{align*}
Using the fact that $\tilde{Z}_i^{(l)} \eqd \tilde{Z}_i$, we get
\begin{align*}
&  \exptn \bigg[  \int_{\real} \sum_{l=1}^{\tilde{U}_1} \sum_{\stackrel{A \in \pow([\tilde{U}_1]) \setminus \{\emptyset\}}{l\in A}} \bigg( 1 - \exp \bigg\{ - \tilde{Z}_i g(x) \bigg\} \bigg)  \nonumber \\
&\hspace{3cm} \prob \Big( Z_1 > 0 \Big) \Big(\prob(Z_i >0 ) \Big)^{|A|} \Big( \prob(Z_i =0) \Big)^{\tilde{U}_1 - |A|}  \nu_\alpha(\dtv x) \bigg]
\end{align*}
as $l \in A \subset [\tilde{U}_1]$. We would like to interchange the integral and the expectation, to get
\begin{align*}
&\exptn \bigg[  \int_{\real} \bigg( 1 - \exp \bigg\{ - \tilde{Z}_i g(x) \bigg\} \bigg) \prob(Z_i >0 ) \sum_{t=1}^\infty \frac{\prob(Z_1 =t)}{\prob(Z_1 >0)} \sum_{l=1}^{t} \sum_{A \setminus \{l\} \in \pow([t] \setminus \{l\}) }   \nonumber \\
& \hspace{3cm} \prob \Big( Z_1 > 0 \Big) \Big(\prob(Z_i >0 ) \Big)^{|A|-1} \Big( \prob(Z_i =0) \Big)^{t - |A|}  \nu_\alpha(\dtv x) \bigg].
\end{align*}
Next we  use the fact that the number of subsets of $[t]$ containing $l$ is the same as the number of all  subsets of $[t-1]$, to get
\begin{align}
&\exptn \bigg[  \int_{\real} \bigg( 1 - \exp \bigg\{ - \tilde{Z}_i g(x) \bigg\} \bigg) \prob(Z_i >0 ) \sum_{t=1}^\infty \frac{\prob(Z_1 =t)}{\prob(Z_1 >0)} \sum_{l=1}^{t} \sum_{A  \in \pow([t-1] ) }   \nonumber \\
& \hspace{3cm} \prob \Big( Z_1 > 0 \Big) \Big(\prob(Z_i >0 ) \Big)^{|A|} \Big( \prob(Z_i =0) \Big)^{t - |A|}  \nu_\alpha(\dtv x) \bigg]  \nonumber \\
&=\exptn \bigg[  \int_{\real} \bigg( 1 - \exp \bigg\{ - \tilde{Z}_i g(x) \bigg\} \bigg) \prob(Z_i >0 ) \sum_{t=1}^\infty \frac{\prob(Z_1 =t)}{\prob(Z_1 >0)} t  \prob \Big( Z_1 > 0 \Big)   \nu_\alpha(\dtv x) \bigg] \nonumber \\
&= \mu \exptn \bigg[  \int_{\real} \bigg( 1 - \exp \bigg\{ - \tilde{Z}_i g(x) \bigg\} \bigg) \prob(Z_i >0 )    \nu_\alpha(\dtv x) \bigg]
\end{align}
using the fact that
\alns{
\sum_{A \in \pow([t])} \Big(\prob(Z_i>0)\Big)^{|A|} \Big( \prob(Z_i=0) \Big)^{t- |A|} = 1.
}
Hence in the i.i.d.\ case the Laplace functional of the limiting random measure is
\aln{
& \estar \bigg[  \exp \bigg\{ - \frac{1}{\mu} W \sum_{i=0}^\infty \frac{1}{\mu^i} \mu \exptn \bigg[  \int_{\real} \bigg( 1 - \exp \bigg\{ - \tilde{Z}_i g(x) \bigg\} \bigg) \prob(Z_i >0 )    \nu_\alpha(\dtv x) \bigg]  \bigg\}\bigg]  \nonumber \\
&= \estar \bigg[  \exp \bigg\{ -  W \sum_{i=0}^\infty \frac{1}{\mu^i} \exptn \bigg[  \int_{\real} \bigg( 1 - \exp \bigg\{ - \tilde{Z}_i g(x) \bigg\} \bigg) \prob(Z_i >0 )    \nu_\alpha(\dtv x) \bigg]  \bigg\} \bigg],
}
which is the same as obtained in Theorem 2.1 in \citet{bhattacharya:hazra:roy:2014}. For the SScDPPP-representation, we refer the reader to \citet{bhattacharya:hazra:roy:2014}. \qedhere

\end{proof}

\subsection{Bounded Offspring Distribution}

Assume $Z_1 \le B$ almost surely. In this case, replace $\lambda(\cdot)$ by $\lambda^{(B)}(\cdot)$, which is supported on $\real^{B}$. Following Theorem 6.1 (Page 173) from \citet{resnick:2007} on multivariate regular variation, it is clear that $\lambda^{(B)}(\cdot)$ on $\real^B$ can be written as the product measure $cm_\alpha \otimes \Theta$ on $(0,\infty] \times S^{B-1}$. Here, $S^{B-1} = \{\mathbf{x} \in \real^B : \norm{x} =1\}$ for any norm $\norm{\cdot}$ on $\real^{B}$ and, for every $x>0$, $cm_\alpha((x,\infty]) =c x^{-\alpha}$ with $c>0$ suitably chosen so that $\Theta(\cdot)$ becomes a probability measure on $S^{B-1}$. The measure $\Theta$ is called angular measure (see Remark 6.2 of the \citet{resnick:2007}). With the help of these measures, which arise naturally in multivariate extreme value theory, we get an explicit SScDPPP-representation in the case when $Z_1 \le B$ as described below.

\begin{cor} \label{cor:bounded_offspring}
Assume $Z_1 \le B$ almost surely. Then under the assumptions of Theorem \ref{thm:mainthm2},
$$N_* \sim \mbox{SScDPPP}(cm_\alpha, D, (sW)^{1/\alpha}),$$
where $D \eqd \sum_{k=1}^{V_1} T_{1k} \delta_{\eta_k}$ with $(V_1, \mathbf{T}_1)=(V_1, (T_{11}, T_{12}, \ldots))$ as described in Subsection~\ref{subsec:brw} and $\pmb{\eta}:= (\eta_1, \ldots,\eta_B)$ has law $\Theta$ on $S^{B-1}$.
\end{cor}

\begin{proof}
It is clear that a Poisson random measure $\poi$ on $\real^B$ with intensity measure $\lambda^{(B)}$ admits the following representation
\alns{
\poi \eqd \sum_{l=1}^\infty \mbfs_{j_l} \delta_{\pmb{\eta}_l},
}
where $\{\pmb{\eta}_l : l \ge 1\}$ are independent copies of the random variable $\pmb{\eta}$, $\sum_l \delta_{j_l} \sim PRM(c m_\alpha)$, and the collections $\{\pmb{\eta}_l \}$ and $\{j_l \}$ are independent.
%Suppose that $V$ be a positive integer valued random variable with following probability mass function
%\aln{
%\prob(V=v) = \frac{1}{s} \prob(Z_1 =v) \sum_{i=0}^\infty \frac{1}{\mu^i} \Big(1- \prob(Z_i =0)^v \Big), \;\; v \in \bbn,
%}
%where $s$ is the normalising constant.
%%\aln{
% %s = \sum_{t=1}^\infty   \prob(Z_1 =t) \sum_{i=0}^\infty \frac{1}{\mu^i} \Big(1- \Big(\prob(Z_i =0)\Big)^t \Big).                                 \label{eq:pmf_V2}
%%}
%Suppose $\bld{T}$ is an $\bbn_0^\bbn$-valued random variable and its probability mass function conditioned on $V$ is given as follows:
%\aln{
%\prob(\bld{T}= \bld{y} | V=v) = \begin{cases} 0 & \mbox{ if } y_k>0  \mbox{ for some } k>v \mbox{ or } \bld{y} =\bld{0}, \\
%\frac{1}{s_v} \sum_{i=0}^\infty \frac{1}{\mu^i} \prod_{m=1}^v \prob(Z_i = y_m ) & \mbox{ otherwise, }\end{cases}
%}
%where $\bld{y}=(y_1, y_1, \ldots) \in \bbn^\infty$, $v \in \bbn$, and $s_v$ is the normalising constant.
%\aln{
%s_t = \sum_{(y_1, \ldots, y_t) \neq \bld{0}_t} \sum_{i=0}^\infty \frac{1}{\mu^i} \prod_{m=1}^t  \prob(Z_i =y_m) \label{eq:cond_pmf_T2}
%}
%with $\bld{0}_t$ being the zero element in $\bbn^t$.
From the calculation of the Laplace functional of $N_*^{(K,B)}$ (see \eqref{eq:defn_nskb} above) it transpires that, in this setup,
\aln{
N_* \eqd \mbfs_{(sW)^{1/\alpha}} \sum_{l=1}^\infty \mbfs_{j_l} D_l,\label{eq:sscdppp_rep_bounded_branching}
}
where $\{D_l : l \ge 1\}$ is a collection of  independent copies of the point process $D$.  This completes the proof.\qedhere
\end{proof}

\appendix \section{List of Notations} \label{sec:notation}
\makeatletter{}
\label{pg:notation} To ease the reading, we list the important notions and notations used in this paper, and the corresponding page numbers. 

{\footnotesize
\begin{center} \renewcommand{\arraystretch}{1.2}
\begin{longtable}{p{2.4cm}p{11.1cm}p{1.5cm}}
  \textbf{Notation}&\textbf{Description}&\textbf{Page} \\

$\bbs_0$ &  $\bbs_0 =\bbs \setminus \{s_0\}$ where $s_0 \in \bbs$, a Polish space & \h \pageref{not:s_0} \\

$\scrm(\bar{\real}_0)$ & Space of all Radon point measures on $\bar{\real}_0$ & \h \pageref{not:mr_0} \\

$\regvar(\bbs_0, \alpha, \lambda)$ & Regularly variation on the space $\bbs_0$ & \h\pageref{not:regvar_ams}\\

 $\nu_\alpha(\cdot)$ & Measure on $\bar{\real}_0$ & \h \pageref{eq:nualpha}\\

 $\lamiid(\cdot)$     & Measure on $\rnz$ & \h \pageref{eq:iidprocess_lambda} \\

 $\rnz$ & $\real^\bbn \setminus \mathbf{0}_\infty$ where $\mathbf{0}_\infty \in \real^\bbn$ with all its components as $0$ & \h \pageref{eq:iidprocesslim} \\

$\bar{\real}_0$ & $[-\infty,\infty] \setminus \{0\}$ & \h \pageref{not:r_0}\\

$\scrm_0$  & $\scrm(\bar{\real}_0) \setminus \{ \nullm \}$ & \h  \pageref{not:scrm_0} \\

$\mbfs_b$ & Scalar multiplication operator for elements of $\scrm_0$ with $b>0$ & \h \pageref{eq: scalar}\\

 $\stas$ & Strictly $\alpha$-stable point process & \h \pageref{not:stas} \\

 $\emptyset$ & Null measure & \h \pageref{not:nullm}\\

 $\rm{ScDPPP}$ & Scale-decorated Poisson point process & \h \pageref{not:scdppp} \\

 $\rm{SScDPPP}$ & Randomly scaled scale-decorated Poisson point process & \h \pageref{not:sscdppp}\\

 $\Iv$ & The unique geodesic path from the root to the vertex $\uv$ & \h \pageref{not:gen_uv}\\

 $|\uv|$ & Generation of the vertex $\uv$ & \h \pageref{not:gen_uv}\\

 $\poi$ & Poisson random measure on $\bar{\real}_0$ & \h \pageref{eq:defn_poi}\\

$\ckr$ & Space of all nonnegative continuous functions on $\bar{\mathbb{R}}_0$ with compact support  & \h \pageref{not:ckr}\\

$\Psi_N(\cdot)$ & Laplace functional of the point process $N$ & \h \pageref{eq:lap_not} \\

$\nu(f)$ & $\int f \dtv \nu$ &  \h \pageref{not:int_measure} \\

$[g]_{sc}$ & $\{f \in \ckr : f = \mbfs_y g \mbox{ for some } y>0\}$ & \h \pageref{not:g_sc}\\

$\Psi_N(\cdot \| \cdot)$ & Scaled Laplace functional & \h \pageref{not:scale_laplace}\\

$\fralpha(\cdot)$ & Frech\'{e}t distribution function & \h \pageref{not:fralpha}\\

$D_n$ & $\{\uv \in \mathbb{V}: |\uv|=n\}$ & \h \pageref{not:D_n}\\

 $0_t \in \real^t$ & The zero vector in $\real^t$, $t \in \bbn \cup \{\infty\}$& \h \pageref{not:zero_b}\\

 $\kappa_\lambda$ & A constant based on the measure $\lambda$ & \h \pageref{not:kappa_lambda} \\

  \end{longtable}
\end{center}
}

\section*{Acknowledgment}
The authors are thankful to Antar Bandyopadhyay and Jean Bertoin for asking a question that resulted in Theorem~\ref{thm:mainthm2} of this paper, and to Frank den Hollander for reading the first draft carefully and giving valuable suggestions. Numerous useful discussions with all of them are gratefully acknowledged.

\bibliographystyle{authordate1}			% include the bibliography
%\bibliography{ayanbib}

\vspace{0.5cm}

\end{document}

%% file: notation.tex
\newcommand{\pointl}{\mathcal{L}}
\newcommand{\pointq}{\mathcal{Q}}
\newcommand{\scrm}{\mathscr{M}}

\newcommand{\inv}{{-1}}
\newcommand{\bbs}{\mathbb{S}}
\newcommand{\cals}{\mathcal{S}}

\newcommand{\bbn}{\mathbb{N}}

\newcommand{\real}{\mathbb{R}}
\newcommand{\uv}{\mathbf{v}}
\newcommand{\uu}{\mathbf{e}}
\newcommand{\Xu}{X_{\uu}}
\newcommand{\Sv}{S_{\uv}}
\newcommand{\Iv}{I_\uv}
\newcommand{\nullm}{\emptyset}
\newcommand{\rnz}{\real_0^{\mathbb{N}}}

\newcommand{\bbt}{\mathbb{T}}

\newcommand{\h}{\hspace{.25cm}}

\newcommand{\scrc}{\mathscr{C}}

\newcommand{\bbm}{\mathbb{M}}
\newcommand{\stas}{\mbox{St$\alpha$S}}

\newcommand{\beq}{\begin{equation}}
\newcommand{\eeq}{\end{equation}}
\newcommand{\alns}[1]{\begin{align*}#1\end{align*}}
\newcommand{\aln}[1]{\begin{align} #1 \end{align}}
\newcommand{\been}{\begin{enumerate}}
\newcommand{\een}{\end{enumerate}}
\newcommand{\norm}[1]{\| #1 \|}

\newcommand{\dnkbo}{N_n'^{ (K,B,1)}}
\newcommand{\tub}{\tilde{U}^{(B)}}
\newcommand{\ub}{U^{(B)}}
\newcommand{\eps}{\epsilon}

\newcommand{\eqd}{\,{\buildrel d \over =}\,}
\newcommand{\ckr}{C_c^+ (\bar{\real}_0)}

\newcommand{\zrer}{\bar{\real}_0}

\newcommand{\point}{\mathcal{P}}
\newcommand{\malpha}{m_\alpha}

\newcommand{\bbo}{\mathbbm{1}}

\newcommand{\tnk}{\tilde{N}_n^{(K)}}

\newcommand{\scrmrz}{\scrm_0}
\newcommand{\aub}{A_\uu^{(B)}}
\newcommand{\nkb}{\tilde{N}_n^{(K,B)}}
\newcommand{\calf}{\mathcal{F}}
\newcommand{\dnkb}{N_n'^{ (K,B)}}
\newcommand{\poi}{\mathscr{P}}

\newcommand{\nskb}{N_*^{(K,B)}}
\newcommand{\bld}{\mathbf}

\DeclareMathOperator{\fralpha}{\Phi_\alpha}

\DeclareMathOperator{\dtv}{d}
\DeclareMathOperator{\pow}{Pow}
\DeclareMathOperator{\eff}{EFF}
\DeclareMathOperator{\regvar}{RV}
\DeclareMathOperator{\lamiid}{\lambda_{iid}}

\usepackage{cleveref}
\newtheorem{thm}{Theorem}[section]
\newtheorem{propn}[thm]{Proposition}
\newtheorem{lemma}[thm]{Lemma}

\newtheorem{cor}[thm]{Corollary}

\theoremstyle{remark}
\newtheorem{remark}[thm]{Remark}

\newtheorem{example}{Example}[section]

\theoremstyle{definition}
\newtheorem{defn}[thm]{Definition}

\newtheorem{ass}[thm]{Assumptions}

\DeclareMathOperator{\prob}{\mathbf{P}}
\DeclareMathOperator{\exptn}{\mathbf{E}}

\DeclareMathOperator{\pstar}{\prob^*}
\DeclareMathOperator{\estar}{\exptn^*}
\DeclareMathOperator{\mbfs}{\mathbf{S}}

\newcommand{\hlconv}{\stackrel{\mbox{\tiny{HL}}}{\longrightarrow}}